\documentclass[11pt]{amsart}

\usepackage{color}

\usepackage{amsmath,amssymb,amsthm,amsfonts,amsbsy,amstext}

\usepackage{amsmath, amssymb} 
\usepackage{amsthm, amsfonts, mathrsfs}

\usepackage{amsfonts,graphicx}

\newcommand{\DD}{\textnormal{D}}

\theoremstyle{plain}
\newtheorem{Theo}{Theorem}[section]
\newtheorem{lem}[Theo]{Lemma}
\newtheorem{cor}[Theo]{Corollary}
\newtheorem{prop}[Theo]{Proposition}

\theoremstyle{plain} \theoremstyle{definition}
\newtheorem{defi}[Theo]{Definition}

\theoremstyle{remark}
\newtheorem{Rema}[Theo]{Remark}
\newtheorem*{rema*}{Remark}

\newcommand{\ZZ}{\mathbb{Z}}

\newcommand{\NN}{\mathbb{N}}

\newcommand{\RR}{\mathbb{R}}

\numberwithin{equation}{section}

\date{}

\begin{document}

\title[Global well-posedness for Boussinesq System]
{Global existence and uniqueness for a non linear Boussinesq system in dimension two}
\author[Samira SULAIMAN]{SAMIRA SULAIMAN}
\address{IRMAR, Universit\'e de Rennes 1\\ Campus de Beaulieu\\ 35~042 Rennes cedex\\ France}
\email{samira.sulaiman@univ-rennes1.fr}

\begin{abstract}
We study the global well-posedness of a two-dimensional Boussinesq system which couples the incompressible Euler equation for the velocity and a transport equation with fractional diffusion of type $\vert \DD\vert^{\alpha}$ for the temperature. 
We prove that for $\alpha>1$  there exists a unique global solution for initial data with critical regularities. 
\end{abstract}

\maketitle

\section{Introduction} In this paper, we study the two-dimensional Euler-Boussinesq system describing the phenomenon of convection in an incompressible  fluid . This system is composed of the \textit{Euler} equations coupled with a transport-diffusion equation governing the evolution of the density. This system is given by
\begin{equation}\label{T1} 
\left\{\begin{array}{ll} 
\partial_{t}v+v\cdot\nabla v+\nabla p=F(\theta) \\ 
\partial_{t}\theta+v\cdot\nabla\theta+\kappa\vert \DD\vert^{\alpha}\theta=0\\
\textnormal{div}v=0\\
v_{| t=0}=v^{0}, \quad \theta_{| t=0}=\theta^{0}.  
\end{array} \right. 
\end{equation}
Above, $v=v(x,t) \in\RR^2$, denotes the velocity vector-field, $p$ is the pressure and  $\theta= \theta(x,t)$ is the temperature . The function $F (\theta)= (F_{1}(\theta),F_{2}(\theta))$ is a vector -valued function such that $F \in \mathcal{C}^{5}$ and $F(0)=0$, $\alpha$ is a real number in $]0,2]$ and the nonnegative parameter $\kappa$ denotes the molecular diffusivity. The fractional Laplacian  $\vert \DD\vert^{\alpha}$ is defined in a standard way through its Fourier transform by 
$$\widehat{\vert \DD\vert^{\alpha}f(\xi)}=\vert \xi \vert^{\alpha}\widehat{f(\xi)}.$$
The system \eqref{T1} generalizes the classical Boussinesq system where $F(\theta)=(0,\theta).$ It is a special case of a class of generalized Boussinesq system introduced  in \cite{YB}.\\
Before discussing the mathematical aspects of our model with general $F$ we will first focus on the special case $F(\theta)=(0,\theta)$ and review the most significant contributions in the theory of global existence and uniqueness. 
We note that in space in dimension two the vorticity $\omega=\partial_{1}v^{2}-\partial_{2}v^{1}$ satisfies the transport-diffusion equation 
$$\partial_{t}\omega+v\cdot\nabla\omega=\partial_{1}\theta.$$
It is well-known that a Beale-Kato-Majda criterion \cite{BKM} can be applied to our model and thus the control of the vorticity in $L^\infty$ space is a crucial step to get  global well-posedness results with smooth initial data. Now, by applying a maximum principle we get
$$\Vert\omega(t)\Vert_{L^{\infty}}\le\Vert\omega^{0}\Vert_{L^{\infty}}+ \int^{t}_{0}\Vert\partial_1\theta(\tau)\Vert_{L^{\infty}}d\tau.$$  
The difficulty is then reduced to estimate the quantity $\int^{t}_{0}\Vert\partial_{1}\theta(\tau)\Vert_{L^{\infty}}d\tau$ and for this purpose the use the the smoothing effects of the transport-diffusion equation is crucial especially for sub-critical dissipation, that is $\alpha>1.$  \\
For the full viscous equations i.e when $\kappa>0$ and $\alpha=2$, the  global well-posedness problem  is solved recently in a series of papers \cite{Cha, dp1, hk2}. 

In \cite{Cha}, Chae proved the global existence and uniqueness for initial data $(v^0,\theta^0)\in H^s\times H^{s}$ with $s>2$. This result was extended in \cite{hk1} by Hmidi and Keraani to initial data $v^0\in B_{p,1}^{\frac{2}{p}+1}$ and $\theta^0\in B_{p,1}^{\frac{2}{p}-1}\cap L^r,$ with  $r\in]2,\infty[$. Recently the study of global existence of Yudovich solutions  has been done in \cite{dp1}.\\
For more  weaker dissipation that is  $1\le\alpha<2$, global well-posedness results have been recently obtained. Indeed,  the subcritical case  $1<\alpha$ was solved by Hmidi and Zerguine \cite{HZ} with critical regularities i.e  $v^{0}\in B^{1+\frac{2}{p}}_{p,1}$ and $\theta^{0}\in B^{-\alpha+1+\frac{2}{p}}_{p,1}\cap L^{r}\;,\,\frac{2}{\alpha-1}<r<\infty.$ They used the maximal smoothing effects for a transport-diffusion which can be roughly speaking summarized as follows: for every $0<\varepsilon<1$
$$
\|\theta\|_{L^1_tC^{1-\varepsilon}}\le C_0(1+t+\|\omega\|_{L^1_tL^\infty}),
$$
with $C_0$ a constant depending on the size of the initial data.
The critical case $\alpha=1,$ is more subtle because the dissipation has the same rate as the possible amplification of the vorticity by $\partial_1\theta$. In \cite{hkr}, Hmidi, Keraani and Rousset gave a positive answer for global wel-posedness by using a hidden cancellation given by the coupling.\\
The main goal of this paper is to extend the results of \cite{HZ} for general source term $F(\theta).$Our result reads as follows (see section 2 for the definitions and the basic properties of Besov spaces).   
\begin{Theo}\label{TX} Let $(\alpha, p)\in]1, 2]\times]1, \infty[$, $v^0\in B_{p, 1}^{1+\frac{2}{p}}$ be a divergence free vector-field of $\RR^2$, $ \theta^0\in B_{p, 1}^{-\alpha+1+\frac{2}{p}}\cap L^\infty$ and $F \in \mathcal{C}^{5}(\RR,\RR).$ 
Then there exists a unique global solution $(v, \theta)$ for the system \eqref{T1} such that
\begin{equation*}
v \in\mathcal{C}\big(\RR_+; B_{p, 1}^{1+\frac{2}{p}}\big)\quad and\quad \theta\in L_{loc}^{\infty}\big(\RR_+; B_{p, 1}^{-\alpha+1+\frac{2}{p}}\cap L^\infty\big)\cap L_{loc}^1\big(\RR_+; Lip \big)
\end{equation*}
\end{Theo} 

If we take $\theta=0,$ then the system \eqref{T1} is reduced to the well-known 2D incompressible Euler system. It is well known that this system is globally well-posed in $H^s$ for $s > 2.$ The main argument  is the BKM criterion \cite{BKM} ensuring that the development of finite -time singularities is related to the blow-up of the $L^\infty$ norm of the vorticity and in that case the vorticity is only transported  by the flow. However the global persistence of  critical Besov regularities $v^0\in B^{1+\frac{2}{p}}_{p,1}$ can not be derived from BKM criterion from.  This problem was solved in \cite{vis} by 
Vishik  and his crucial tool is a new  logarithmic estimate which can be formulated as follows:
$$\Vert f\circ g^{-1}\Vert_{B^{0}_{\infty,1}}\le C\big(1+\log(\Vert\nabla g\Vert_{L^{\infty}}\Vert\nabla g^{-1}\Vert_{L^{\infty}})\big)\Vert f \Vert_{B^{0}_{\infty,1}},$$ 
with $f\in B^{0}_{\infty,1}$, $g$ is a $C^1$- diffeomorphism  preserving Lebesgue measure and $C$ some constant depending only on the dimension $d$ (see Theorem 4.2 in \cite{vis} p-209 for the proof). 

\begin{Rema}
 In the above theorem, we can take $F \in \mathcal{C}^{[1+\frac{2}{p}]+2}$ instead of $F \in \mathcal{C}^{5}.$\\
On the other hand if we assume that $F \in \mathcal{C}_{b}^{5}$, then we can replace the assumption $\theta^{0}\in L^{\infty}$ by $\theta^{0}\in L^{r}\;,\;\frac{2}{\alpha-1}<r<\infty.$ 
\end{Rema}
Let us now discuss briefly the difficulties that one has to deal with.  The formulation  vorticity-temperature of the system \eqref{T1} is described by, 

\begin{equation*} 
\left\{\begin{array}{ll} 
\partial_{t}\omega+v\cdot\nabla\omega=\partial_{1}(F_{2}(\theta))- \partial_{2}(F_{1}(\theta)) \\ 
\partial_{t}\theta+v\cdot\nabla\theta+\vert \DD\vert^{\alpha}\theta=0\\
\omega_{|t=0}=\omega^{0}, \quad \theta_{|t=0}=\theta^{0}.  
\end{array} \right. 
\end{equation*} 
Taking the $L^2$-scalar product we get successively  
$$\Vert \omega(t) \Vert_{L^2}\le \Vert \omega^0 \Vert_{L^2}+\Vert\nabla F\Vert_{L^{\infty}} \displaystyle \int^{t}_{0}\Vert \nabla \theta(\tau) \Vert_{L^2}d\tau$$
and 
$$\Vert\theta(t)\Vert^2_{L^2}+\Vert\theta(t)\Vert^2_{L_{t}^{2}\dot{H}^{\frac{\alpha}{2}}}\le\Vert \theta^0\Vert^2_{L^2}.$$
We observe that one can take benefit of these estimates only for $\alpha=2$ in which case   we get a bound for $\omega$ in $L^{\infty}_{loc}(\RR_+,L^2)$ and a bound for $\theta$ in $L_{loc}^{\infty}(\RR_+,L^{2})\cap L^{2}_{loc}(\RR_+,\dot{H}^{1}).$ However for $1<\alpha<2,$ there is no obvious a priori estimates  for the vorticity and  we will use the  idea developed in \cite{HZ}  consisting in the  use   of  the maximal smoothing effect of the transport-diffusion equation.

The plan of the rest of this paper is organized as follows. In section \ref{W} we detail some  basic notions of Littlewood-Paley theory, function spaces and we recall some useful lemmas. We prove in section \ref{W1} some smoothing effects about a transport-diffusion equation which we  need for the proof of our main result. The proof of our main result is given in section \ref{W2}. Finally we prove in an appendix the generalized Bernstein inequality.

\section{SOME DEFINITION AND TECHNICAL TOOLS}\label{W}
In this preliminary section, we are going to recall the so-called Littlewood-Paley operators and give some of their elementary properties. It will be also convenient to introduce some function spaces and review some important lemmas that will be used later.\\ 
We denote by $C$ any positive constant than will change from line to line and $C_{0}$ a real positive constant depending on the size of the initial data. We will use the following notations:\\
$\bullet$ For any positive $A$ and $B$, the notation  $A\lesssim B$ means that there exists a positive constant $C$ such that $A\leqslant CB$.\\
$\bullet$ We denote by $\dot{W}^{1,p}$ with $1\le p\le \infty$ the space of distribution $f$ such that $\nabla f \in L^p$ (see section \ref{W2}).\\
First of all, we define the dyadic decomposition of the full space $\mathbb{R}^{d}$ and recall the Littlewood-Paley operators (see for example \cite{che}).\\
There exists two nonnegative radial functions $\chi\in\mathcal{D}(\mathbb{R}^{2})$ and $\varphi\in\mathcal{D}(\mathbb{R}^{2}\backslash\{ 0\})$ such that 
\begin{enumerate}
\item$\chi(\xi)+ \displaystyle \sum_{q\ge 0}\varphi(2^{-q}\xi)=1, \quad\forall \xi\in\mathbb{R}^{2},$
\item$\displaystyle\sum_{q\in\mathbb{Z}}\varphi(2^{-q}\xi)=1, \quad\forall \xi\in\mathbb{R}^{2}\backslash\{0\},$
\item$\vert p-q\vert\ge 2\Rightarrow\mbox{supp }{\varphi}(2^{-p}\cdot)\cap\mbox{supp }{\varphi}(2^{-q}\cdot)=\varnothing,$
\item$q\ge 1\Rightarrow \mbox{supp }{\chi}\cap\mbox{supp }{\varphi}(2^{-q}\cdot)=\varnothing.$
\end{enumerate}
Let $h=\mathcal{F}^{-1}\varphi$ and $\bar{h}=\mathcal{F}^{-1}\chi,$ the frequency localization operators $\Delta_{q}$ and $S_{q}$ are defined by
\begin{eqnarray*}\label{d}
\Delta_{q}f&=&\varphi(2^{-q}\DD)f= 2^{2\,q}\int_{\RR^{2}}h(2^{q}y)f(x-y)\,dy\;\;\;\textnormal{for}\quad q\geqslant 0,\\
S_{q}f&=&\chi(2^{-q}\DD)f =\displaystyle \sum_{-1\le p\le q-1}\Delta_{p}f= 2^{2\,q} \int_{\RR^{2}}\bar{h}(2^{q}y)f(x-y)\,dy,\\
\Delta_{-1}f&=&S_{0}f ,\qquad \Delta_{q}f=0 \qquad \textnormal{for}\quad q\le-2.
\end{eqnarray*}
It may be easily checked that $$f=\sum_{q\in \ZZ}\Delta_{q}f,\;\;\forall f \in \mathcal{S}^{\prime}(\RR^{2}).$$
Moreover, the Littlewood-Paley decomposition satisfies the property of almost orthogonality:
$$\Delta_{p}\Delta_{q}f=0\qquad \textnormal{if} \qquad \vert p-q \vert \geqslant 2 \qquad$$
$$\Delta_{p}(S_{q-1}\Delta_{q}f)=0 \qquad \textnormal{if} \qquad  \vert p-q \vert \geqslant 5.$$\\
Let us note that the above operators $\Delta_{q}$ and $S_{q}$ map continuously $L^{p}$ into itself uniformly with respect to  $q$ and $p$. 
We will need also the homogeneous operators : 
$$\forall q\in\mathbb{Z}\quad\dot{\Delta}_{q}v=\varphi(2^{-q}\DD)v\quad \textnormal{and} \quad \dot{S}_{q}v= \sum_{ p\le q-1}\dot{\Delta}_{p}v.$$
We notice that $\Delta_{q}= \dot{\Delta}_{q}\;,\forall\;q \in \NN$ and $S_{q}$ coincides with $\dot{S_{q}}$ on tempered distributions modulo polynomials.\\

We now give the way how the product acts on Besov spaces. We shall use the dyadic decomposition.\\
Let us consider two tempered distributions $u$ and $v,$ we write
\begin{equation*}
u= \sum_{q}\Delta_{q}u \qquad \textnormal{and} \qquad v=\sum_{q^{\prime}}\Delta_{q^{\prime}}v
\end{equation*}
\begin{equation*}
uv= \sum_{q,q^{\prime}}\Delta_{q}u \Delta_{q^{\prime}}v.
\end{equation*}
Now, let us introduce Bony's decomposition see \cite{bo}.
\begin{defi}\label{def1}
We denote by $T_{u}v$ the following bilinear operator :
\begin{equation*}
T_u v=\displaystyle \sum_{q}S_{q-1}u\Delta_q v.
\end{equation*}
The remainder of $u$ and $v$ denoted by $R(u,v)$ is given by the following bilinear operator :
\begin{equation*} 
R(u,v)=\displaystyle \sum_{\vert q-q^\prime\vert \leqslant 1}\Delta_qu\Delta_{q^\prime }v.
\end{equation*}
\end{defi}
Just by looking at the definition,it is clear that
\begin{equation*}
uv=T_u v+T_v u+R(u,v).
\end{equation*}
 With the introduction of $\Delta_{q}$, let us recall the definition of Besov space, see \cite{che}.
\begin{defi}\label{def2}
Let $s\in\RR$ and $1 \le p,r \le +\infty.$ The inhomogeneous Besov space $B_{p,r}^s$ is defined by
$$B^{s}_{p,r}=\left\lbrace f \in \mathcal{S}^{\prime}(\RR^{2}) : \Vert f \Vert_{B^{s}_{p,r}}<\infty \right\rbrace .$$
Here
$$\|f\|_{B_{p,r}^s}:=\| 2^{qs} \|\Delta_q f\|_{L^{p}}\|_{\ell ^{r}}.$$
We define also the homogeneous norm 
$$\|f\|_{\dot B_{p,r}^s}:=\Vert (2^{qs}
\|\dot\Delta_q f\|_{L^{p}})_{q}\Vert_{\ell ^{r}(\ZZ)}.$$
\end{defi}
The definition of Besov spaces does not depend on the choice of the dyadic decomposition. 
The two spaces $H^s$ and $B^{s}_{2,2}$ are equal and we have
\begin{equation*}
\dfrac{1}{C^{\vert s \vert+1}}\Vert u \Vert_{B^{s}_{2,2}}\le \Vert u \Vert_{H^{s}}\le C^{\vert s \vert+1}\Vert u \Vert_{B^{s}_{2,2}}.
\end{equation*}

Our study will require the use of the following coupled spaces.
Let $T>0$ and $\rho\geq1,$ we denote by $L^\rho_{T}B_{p,r}^s$ the space of distributions $f$ such that 
$$\|f\|_{L^\rho_{T}B_{p,r}^s}:= \Big\|\Big( (2^{qs} \|\Delta_q f\|_{L^p})_{q}\Big)_{\ell ^{r}}\Big\|_{L^\rho_{T}}<+\infty.$$
Besides the usual mixed space $L^\rho_{T}B_{p,r}^s,$ we also need Chemin-Lerner space $\widetilde L^\rho_{T}{B_{p,r}^s}$ which defined as the set of all distributions $f$ satisfying  
 $$\|f\|_{ \widetilde L^\rho_{T}{B_{p,r}^s}}:= \Vert (2^{qs}
\|\Delta_q f \|_{L^\rho_{T}L^p})_{q}\Vert_{\ell ^{r}}<+\infty .$$\\
The relation between these spaces are detailed in the following lemma, which is a direct consequence of the  Minkowski inequality. 
\begin{lem}\label{lem1}
 Let $ s\in\RR ,\varepsilon>0$ and $(p,r,\rho) \in[1,+\infty]^3.$ Then we have the following embeddings 
$$L^{\rho}_{T}B^{s}_{p,r}\hookrightarrow\widetilde L^{\rho}_{T}B^{s}_{p,r}\hookrightarrow L^{\rho}_{T}B^{s-\varepsilon}_{p,r}\;\;\;\textnormal{if}\quad r\geqslant\rho.$$ 
$${L^\rho_{T}}{B_{p,r}^{s+\varepsilon}}\hookrightarrow\widetilde L^\rho_{T}{B_{p,r}^s}\hookrightarrow L^\rho_{T}B_{p,r}^s\;\;\;\textnormal{if}\quad 
\rho\geq r.$$
\end{lem}
A further important result that will be constantly used here is the so-called Bernstein inequalities (for the proof see \cite{che} and the references therein) which are detailed below.
\begin{lem}\label{lem2}
There exists a constant $C>0$ such that for every $q\in\ZZ\,,\,k \in \NN$ and for every tempered distriubution $u$ we have  
\begin{eqnarray*}
\sup_{\vert\alpha\vert=k}\Vert\partial^{\alpha}S_{q}u\Vert_{L^{b}}\leqslant C^{k}2^{q\big(k+2\big(\frac{1}{a}-\frac{1}{b}\big)\big)}\Vert S_{q}u\Vert_{L^{a}}\quad \textnormal{for}\quad \; b\geqslant a\geqslant 1\\
C^{-k}2^{qk}\Vert\dot{\Delta}_{q}u\Vert_{L^{a}}\leqslant \sup_{\vert\alpha\vert=k}\Vert\partial^{\alpha}\dot{\Delta}_{q}u\Vert_{L^{a}}\leqslant C^{k}2^{qk}\Vert \dot{\Delta}_{q}u\Vert_{L^{a}}.
\end{eqnarray*}
\end{lem}
Notice that Bernstein inequalities remain true if we change the derivative  $\partial^{\alpha}$ by the fractional derivative $\vert \DD\vert^{\alpha}.$
We can find a proof of the next Proposition in \cite{hk1} which is an extension of \cite{vis}.
\begin{prop}\label{prop1}
Let $(p, r)\in[1,\infty]^2$, $v$ be a divergence free vector-field belonging to the space $L^1_{loc}(\RR_{+};\textnormal{Lip}(\RR^2))$ and let $a$ be a smooth solution of the following transport equation,
\begin{equation*}\left\lbrace
\begin{array}{l}
\partial_t a+v\cdot\nabla a=f\\
{a}_{| t=0}=a^{0}.\\
\end{array}
\right.
\end{equation*}
If the initial data  $a^{0}\in B_{p,r}^0,$ then we have for all $t\in\RR_{+}$
\begin{equation*}
\|a\|_{\widetilde L^\infty_{t}B_{p,r}^0}\lesssim\big(\|a^{0}\|_{B_{p,r}^0}+\|f\|_{\widetilde L^1_{t}B_{p,r}^0}\big)\Big(1+ \int_{0}^{t}\|\nabla v(\tau)\|_{L^\infty}d\tau\Big).
\end{equation*}
\end{prop}
We now give the following commutator estimate which proved in \cite{hk1}.
\begin{prop}\label{prop2}
Let $u$ be a smooth function and $v$ be a smooth divergence-free vector field of $\RR^{2}$ with vorticity $\omega:=curl v.$ Then we have for all $q \geqslant -1$
\begin{equation*}
\Vert [ \Delta_{q},v\cdot\nabla]u \Vert_{L^{\infty}}\lesssim \Vert u \Vert_{L^{\infty}}\left(\Vert \nabla \Delta_{-1}v  \Vert_{L^{\infty}}+(q+2) \Vert \omega \Vert_{L^{\infty}}\right).
\end{equation*}
\end{prop}
We recall now the following result of propagation of Besov regularities which is discussed in \cite{che}.
\begin{prop}\label{prop3}
Let $v$ be a solution of the incompressible Euler system,  
$$\partial_t v+v\cdot\nabla v+\nabla p=f,\quad v_{|t=0}=v^0,\quad \textnormal{div}\, v=0.$$
Then for $s>-1, (p,r)\in]1,\infty[\times[1,\infty]$ we have
$$\|v(t)\|_{B_{p,r}^s}\leqslant C e^{CV(t)}\Big(\|v^0\|_{B_{p,r}^s}+\int_0^te^{-CV(\tau)}\|f(\tau)\|_{B_{p,r}^s}d\tau\Big),$$
with $V(t)=\|\nabla v\|_{L^1_tL^\infty}.$
\end{prop}
To finish this paragraph we need the following theorem which give the action of smooth functions on the Besov spaces $B^{s}_{p,r},$ (see \cite{AG} for the proof).
\begin{Theo}\label{theo2}
Let $F\in \mathcal{C}^{[s]+2}\;,\;s$ a positive real number and $F$ vanishing at $0.$ If $u$ belongs to $B^{s}_{p,r}\cap L^{\infty}$, with $(p,r) \in [1,+\infty]^{2},$  then $F \circ u$ belongs to $B^{s}_{p,r}$ and we have
\begin{equation*} 
\Vert F \circ u \Vert_{B^{s}_{p,r}}\leqslant C_{s}\displaystyle \sup_{\vert x \vert \le C \Vert u \Vert_{L^{\infty}}}\Vert F^{[s]+2}(x)\Vert_{L^{\infty}} \Vert u \Vert_{B^{s}_{p,r}}.
\end{equation*}
\end{Theo}

\section{AROUND A TRANSPORT-DIFFUSION EQUATION}\label{W1}

In this section, we will give some useful estimates for any smooth solution of linear transport-diffusion model given by
\begin{equation}\label{T2}
\left\{ \begin{array}{ll} 
\partial_{t}\theta+v\cdot\nabla\theta+\vert {\DD}\vert^{\alpha}\theta=f\\
\theta_{| t=0}=\theta^{0}.
\end{array} \right. 
\end{equation} 
We will discuss two kinds of estimates : $L^{p}$ estimates and smoothing effects.\\
The proof of the following $L^{p}$ estimates can be found in \cite{C-C}.
\begin{lem}\label{lem3} Let $v$ be a smooth divergence free vector-field of $\RR^{2}$ and $\theta$ be a smooth solution of the equation \eqref{T2}. Then 
for every $p\in[1, \infty]$ 
\begin{equation*}
\Vert\theta(t)\Vert_{L^p}\leqslant\Vert\theta^0\Vert_{L^p}+\int_{0}^{t}\Vert f(\tau)\Vert_{L^p}d\tau.
\end{equation*}
\end{lem}
We give now the following smoothing effects which is proved in \cite{HZ}.
\begin{prop}\label{prop4} Let $p\in[1, \infty], s>-1$ and $v$ be a smooth divergence free vector-field of $\RR^2$. Let $\theta$ be a smooth solution of \eqref{T2}, then
\begin{equation*}
 \|\theta\|_{\widetilde L^\infty_{t}B_{p,1}^s}+\|\theta\|_{ L^1_{t}B_{p,1}^{s+\alpha}}\leqslant Ce^{CV(t)}\Big(\|\theta^0\|_{B_{p,1}^s}(1+t)+\|f\|_{L_{t}^{1}B^{s}_{p,1}}+\int_{0}^t\Gamma_s(\tau)d\tau\Big),
\end{equation*}
with,
\begin{equation*} 
V(t)\overset{def}{=}\displaystyle \int_{0}^t\|\nabla v(\tau)\|_{L^\infty}d\tau,\quad\Gamma_s(t)\overset{def}{=} 
\|\nabla\theta(t)\|_{L^\infty}\|v(t)\|_{B_{p,1}^s}1_{[1,\infty[}(s).
\end{equation*} 
\end{prop}
We intend to prove the two following smoothing effects. One is detailed below.
\begin{prop}\label{prop5}
Let $v$ be a smooth divergence free vector-field of $\RR^2$ with vorticity $\omega$ and $\theta$ be a smooth solution of \eqref{T2}. Then for every $r\in[2, \infty[,$ there exists a constant $C$ such that for every $\rho\geq1\;,\;q\in\NN$ and $t\in\RR_+$
\begin{equation*}
2^{q{\frac{\alpha}{\rho}}}\Vert\Delta_{q}\theta\Vert_{L^\rho_tL^r}\le C \Vert \Delta_{q}\theta^0\Vert_{L^r}+ C \Vert \theta^{0}\Vert_{L^{\infty}} \Vert {\rm \omega} \Vert_{L^{1}_{t}L^{r}}+C \Vert f \Vert_{L^{1}_{t}L^{r}}.
\end{equation*}
Moreover for $q=-1,$ we have
\begin{equation*}
\Vert \Delta_{-1}\theta \Vert_{L^{\rho}_{t}L^{r}}\le C t^{\frac{1}{\rho}} \big( \Vert \Delta_{-1}\theta^{0} \Vert_{L^{r}}+\Vert \theta^{0}\Vert_{L^{\infty}} \int^{t}_{0}\Vert \omega(\tau)\Vert_{L^{r}}d\tau+\int^{t}_{0}\Vert f(\tau)\Vert_{L^{r}}d\tau\big).
\end{equation*}
\end{prop}
\begin{proof}
We start with localizing in frequencies the equation of $\theta$: for $q \in\NN$ we set $\theta_{q}:=\Delta_{q}\theta.$ Then
\begin{equation}\label{T3}
\partial_{t}\theta_{q}+v\cdot\nabla\theta_{q}+\vert{\rm D}\vert^{\alpha}\theta_{q}=-\big[\Delta_{q}, v\cdot\nabla\big]\theta+\Delta_{q}f.
\end{equation}
Multiplying the above equation  by $|\theta_q|^{r-2}{\theta}_q,$ integrating by parts and using H\"{o}lder inequality, we get 
$$\frac1r\frac{d}{dt}\|\theta_q\|_{L^r}^r+\int_{\RR^2}(\vert\textnormal{D}\vert^{\alpha}\theta_q) |\theta_q|^{r-2}\theta_{q}dx \le \|\theta_q\|_{L^r}^{r-1}\|[\Delta_q, v\cdot\nabla]\theta\|_{L^r}+\|\theta_q\|_{L^r}^{r-1}\Vert \Delta_{q}f \Vert_{L^r}.$$
Using the following generalized Bernstein inequality see the appendix,
$$\forall 1 < r\;\;,\;c 2^{q\alpha}\Vert\theta_q\Vert_{L^{r}}^{r} \le \int_{\RR^2}(\vert\textnormal{D}\vert^{\alpha}\theta_q)\vert\theta_{q}\vert^{r-2}\theta_{q}dx,$$
where $c$ depends on $r.$ Inserting this estimate in the previous one we obtain 
$$\frac{1}{r}\frac{d}{dt}\|\theta_q(t)\|_{L^r}^r+c2^{q\alpha}\|\theta_q(t)\|_{L^r}^r\lesssim\|\theta_q(t)\|_{L^r}^{r-1}\|[\Delta_q, v\cdot\nabla]\theta(t)\|_{L^r}+\|\theta_q(t)\|_{L^r}^{r-1}\Vert \Delta_{q}f(t)\Vert_{L^r}.$$
This implies that $$\frac{d}{dt}\|\theta_q(t)\|_{L^r}+c2^{q\alpha} \|\theta_q(t)\|_{L^r}\lesssim \|[\Delta_q, v\cdot\nabla]\theta(t)\|_{L^r}+\Vert \Delta_{q}f(t)\Vert_{L^r}.$$\\
We multiply the above inequality by $e^{c t 2^{q \alpha}},$ we find
\begin{equation}\label{T4}
\frac{d}{dt}\Big( e^{ct 2^{q\alpha}}\|\theta_q(t)\|_{L^r}\Big)\lesssim e^{ct 2^{q\alpha}}\|[\Delta_q, v\cdot\nabla]\theta(t)\|_{L^r}+e^{ct 2^{q\alpha}}\Vert\Delta_{q}f(t) \Vert_{L^r}.
\end{equation}
To estimate the right hand-side, we shall use the following Lemma (see \cite{hkr} for the proof of this Lemma).
\begin{lem}\label{lem4}
Let $v$ be a smooth divergence-free vector field and $\theta$ be a smooth scalar function. Then for all $p\in[1,\infty]$ and $q\geqslant-1,$ 
$$\|[\Delta_q, v\cdot\nabla]\theta\|_{L^p}\lesssim \| \nabla v \|_{L^{p}}\|\theta\|_{L^{\infty}}.$$
\end{lem}
Combined with \eqref{T4} this lemma yields
\begin{eqnarray*}\label{d3}
\frac{d}{dt}\Big(e^{ct 2^{q\alpha}}\|\theta_q(t)\|_{L^r}\Big)&\lesssim& e^{ct 2^{q\alpha}}\|\nabla v(t)\|_{L^{r}}\|\theta(t)\|_{L^\infty}+e^{ct 2^{q\alpha}}\Vert f(t) \Vert_{L^r}\\
&\lesssim& e^{ct 2^{q\alpha}}\|\omega(t)\|_{L^{r}}\|\theta^{0}\|_{L^\infty}+e^{ct 2^{q\alpha}}\Vert f(t)\Vert_{L^r},
\end{eqnarray*}
we have used in the last line Lemma \ref{lem3} and the classical fact $$\| \nabla v \|_{L^{p}}\lesssim\|\omega\|_{L^{p}}\quad\forall p\in]1,\infty[.$$
Integrating the differential inequality we get
$$\|\theta_q(t)\|_{L^r}\lesssim \|\theta_q^0\|_{L^r}e^{-ct2^{q\alpha}}+\|\theta^{0}\|_{L^{\infty}}\int_0^t e^{-c(t-\tau)2^{q\alpha}}\|\omega(\tau)\|_{L^r}d\tau+\int_0^t e^{-c(t-\tau)2^{q\alpha}}\Vert f(\tau) \Vert_{L^r}d\tau.$$
By taking the $L^{\rho}[0,t]$ norm and using convolution inequalities we obtain
\begin{eqnarray*}\label{d4}
2^{q \frac{\alpha}{\rho}}\|\theta_{q}\|_{L^{\rho}_{t}L^{r}}&\lesssim& \|\theta_{q}^{0}\|_{L^{r}}+\|\theta^{0}\|_{L^{\infty}}\int_{0}^{t}\|\omega(\tau)\|_{L^{r}}d\tau+\int_{0}^{t}\Vert f(\tau) \Vert_{L^{r}}d\tau\\
&\lesssim& \|\theta_{q}^{0}\|_{L^{r}}+\|\theta^{0}\|_{L^{\infty}}\|\omega\|_{L^{1}_{t}L^{r}}+\Vert f \Vert_{L^{1}_{t}L^{r}}.
\end{eqnarray*}
This is the desired result for $q \in \NN.$\\

For $q=-1,$ we apply the operator $\Delta_{-1}$ to \eqref{T2}, we obtain 
\begin{equation*}
\partial_{t}\Delta_{-1}\theta+v\cdot\nabla\Delta_{-1}\theta+\vert \DD \vert^{\alpha}\Delta_{-1}\theta=-\big[\Delta_{-1}, v\cdot\nabla\big]\theta+\Delta_{-1}f.
\end{equation*}
Taking the $L^{r}$ norm and using Lemmas \ref{lem3} and \ref{lem4} we get
$$\Vert\Delta_{-1}\theta\Vert_{L^{r}}\le\Vert\Delta_{-1}\theta^{0}\Vert_{L^{r}}+\int^{t}_{0}\Vert[\Delta_{-1}, v\cdot\nabla]\theta(\tau)\Vert_{L^{r}}d\tau +\int^{t}_{0}\Vert\Delta_{-1}f(\tau)\Vert_{L^{r}}d\tau$$
$$\qquad\;\;\lesssim\Vert\Delta_{-1}\theta^{0}\Vert_{L^{r}}+\int^{t}_{0}\Vert\nabla v(\tau)\Vert_{L^{r}} \Vert\theta(\tau)\Vert_{L^{\infty}}d\tau+\int^{t}_{0}\Vert f(\tau)\Vert_{L^{r}} d\tau$$
$$\quad\,\lesssim\Vert\Delta_{-1}\theta^{0}\Vert_{L^{r}}+ \Vert\theta^{0}\Vert_{L^{\infty}}\int^{t}_{0}\Vert \omega(\tau)\Vert_{L^{r}}d\tau+\int^{t}_{0}\Vert f(\tau)\Vert_{L^{r}}d\tau.$$
By taking the $L^{\rho}[0,t]$ norm and using H\"older inequality, we obtain finally
\begin{equation*}
\Vert \Delta_{-1}\theta \Vert_{L^{\rho}_{t}L^{r}}\lesssim t^{\frac{1}{\rho}} \big( \Vert \Delta_{-1}\theta^{0} \Vert_{L^{r}}+\Vert \theta^{0}\Vert_{L^{\infty}}\int^{t}_{0}\Vert \omega(\tau)\Vert_{L^{r}}d\tau+\int^{t}_{0}\Vert f(\tau)\Vert_{L^{r}}d\tau\big).
\end{equation*}
This is the desired result.
\end{proof}

The second smoothing effect is given by the following Proposition.
\begin{prop} \label{prop6}
Let $v$ be a smooth divergence free vector-field of $\RR^{2}$ with vorticity $\omega.$ Let $\theta$ be a smooth solution of \eqref{T2}. Then we have for $q \geqslant -1$ and for $t \in \RR_{+}$ with $[f\equiv 0],$
\begin{equation*}
2^{q \alpha}\int^{t}_{0}\Vert \Delta_{q}\theta(\tau) \Vert_{L^{\infty}}d\tau \lesssim \Vert \theta^{0}\Vert_{L^{\infty}}\left(1+t+(q+2) \Vert \omega \Vert_{L^{1}_{t}L^{\infty}}+\Vert \nabla \Delta_{-1}v \Vert_{L^{1}_{t}L^{\infty}}\right). 
\end{equation*}
\end{prop}
\begin{proof}
The idea of the proof will be done in the spirit of \cite{hmi}. First we prove the smoothing effects for a small interval of time depending of vector $v$, but it depend not of the initial data. In the second step, we proceed to a division in time thereby extending the estimate at any time arbitrarily chosen positive.
\subsection{Local estimates}
We localize in frequency the evolution equation and rewriting the equation in Lagrangian coordinates .\\
Let $q\in\NN,$ then the Fourier localized function $\theta_{q}:= \Delta_{q}\theta$ satisfies 
\begin{equation}\label{T5}
\partial_{t}\theta_{q}+S_{q-1}v\cdot\nabla\theta_{q}+\vert {\rm D}\vert^{\alpha}\theta_{q}=(S_{q-1}v-v)\cdot\nabla\theta_q-\big[\Delta_{q}, v\cdot\nabla]\theta:=h_q.
\end{equation}
First, we shall estimate the function $h_{q}$ in the space $L^{\infty}$, for the first term we have 
\begin{eqnarray*}\label{dmlm6}
\Vert(S_{q-1}v-v)\cdot\nabla\theta_{q}\Vert_{L^{\infty}}&\le& \Vert S_{q-1}v-v\Vert_{L^{\infty}}\Vert \nabla\theta_{q}\Vert_{L^{\infty}}\\
&\lesssim& \displaystyle \sum_{j \geqslant q-1}\Vert \Delta_{j}v \Vert _{L^{\infty}}2^{q}\Vert \theta_{q}\Vert_{L^{\infty}}\\
&\lesssim& \Vert \theta^{0} \Vert_{L^{\infty}}\displaystyle \sum_{j \geqslant q-1}2^{q-j}\Vert \Delta_{j}\omega \Vert _{L^{\infty}}\\
&\lesssim& \Vert \theta^{0}\Vert_{L^{\infty}}\Vert \omega \Vert_{L^{\infty}}.
\end{eqnarray*}\\
We have used H\"older inequality, Lemma \ref{lem2} and the result $$\Vert \Delta_{j}v \Vert_{L^{p}}\thickapprox 2^{-j}\Vert \Delta_{j}\omega \Vert_{L^{p}}\;\;,\;\;\forall p \in[1,\infty]\;\;\textnormal{and}\quad j \in \NN.$$
For $\Vert [\Delta_{q},v\cdot\nabla]\theta \Vert_{L^{\infty}},$ we use Lemma \ref{lem3} and Proposition \ref{prop2}, then 
\begin{equation*}
\Vert [\Delta_{q},v\cdot\nabla]\theta \Vert_{L^{\infty}}\lesssim \Vert \theta^{0} \Vert_{L^{\infty}}\bigg(\Vert \nabla \Delta_{-1}v \Vert_{L^{\infty}}+(q+2) \Vert \omega \Vert_{L^{\infty}}\bigg).
\end{equation*}\\
This implies that 
\begin{eqnarray*}\label{d7}
\Vert h_{q}(t)\Vert_{L^{\infty}}&\le& \Vert(S_{q-1}v-v)\cdot\nabla\theta_{q}\Vert_{L^{\infty}}+\Vert [\Delta_{q}, v\cdot\nabla]\theta\Vert_{L^{\infty}}\\
&\lesssim& \Vert \theta^{0}\Vert_{L^{\infty}}\bigg(\Vert \nabla \Delta_{-1}v \Vert_{L^{\infty}}+(q+2)\Vert \omega \Vert_{L^{\infty}}\bigg).
\end{eqnarray*} 
Let us now introduce the flow $\psi_{q}$ of the regularized velocity $S_{q-1}v,$
\begin{equation*}
\psi_{q}(t, x)=x+\int_{0}^{t}S_{q-1}v\big(\tau, \psi_{q}(\tau, x)\big)d\tau.	
\end{equation*}
We set  
\begin{equation*}
\Bar{\theta}_{q}(t, x)=\theta_{q}(t, \psi_{q}(t, x))\quad \mbox{and} \quad \Bar{h}_{q}(t, x)=h_{q}(t, \psi_{q}(t, x)).
\end{equation*}
Then we have 
\begin{equation}\label{T6}
\partial_{t}\Bar{\theta}_{q}+\vert\DD\vert^{\alpha}\Bar{\theta}_{q}=\bar{h}_{q}+\vert \DD\vert^{\alpha}\big(\theta_{q}\circ\psi_{q}\big)-\big(\vert \DD\vert^{\alpha}\theta_{q}\big)\circ\psi_{q}:=\bar{h}_q+h_{q}^1.
\end{equation}
Let us admit the following estimate proven in \cite{HZ}
\begin{equation}\label{T7}
\|\vert \DD\vert^{\alpha}\big(\theta_{q}\circ\psi_{q}\big)-\big(\vert \DD\vert^{\alpha}\theta_{q}\big)\circ\psi_{q}\|_{L^p}\le Ce^{CV_{q}(t)}{V_{q}(t)}2^{\alpha q}\Vert\theta_{q}\Vert_{L^{p}}\;,\;\forall\; p \in[1,\infty],\\\\
\end{equation}
where \qquad\qquad $V_{q}(t)= \displaystyle \int^{t}_{0}\Vert \nabla S_{q-1}v(\tau) \Vert_{L^{\infty}} d\tau.$\\
Now, since  the flow $\psi_{q}$ preserves Lebesgue measure then we get by \eqref{T7} 
\begin{eqnarray*}
\big\Vert{h_q^1}(t)\big\Vert_{L^\infty}\le C e^{CV_{q}(t)}V_{q}(t)2^{\alpha q}\Vert\theta_{q}\Vert_{L^{\infty}}.
\end{eqnarray*}
Now, we will again localize in frequency the equation \eqref{T6} through the operator $\Delta_{j},$
$$\partial_{t}\Delta_j\Bar{\theta}_{q}+\vert \DD\vert^{\alpha}\Delta_j\Bar{\theta}_{q}=\Delta_j\Bar{h}_q+\Delta_{j}h^{1}_{q}.$$
Using Duhamel formula
\begin{equation*}
\Delta_{j}\Bar{\theta}_{q}(t,x)= e^{-t \vert \DD \vert^{\alpha}}\Delta_{j}\theta^{0}_{q}+\int^{t}_{0}e^{-(t-\tau) \vert \DD \vert^{\alpha}}\Delta_{j}\Bar{h}_{q}(\tau) d\tau +\int^{t}_{0}e^{-(t-\tau) \vert \DD \vert^{\alpha}}\Delta_{j}h^{1}_{q}(\tau) d\tau.
\end{equation*}
At this stage we need the following lemma (for $\alpha \in ]0,2[$ see \cite{hk3} ).
\begin{lem}\label{lem5}
There exists a positive constant $C$ such that for any $u \in L^{p}$ with $p \in [1,+\infty],$ for any $t,\alpha \in \mathbb{R_{+}}$ and any $j \in \mathbb{N},$ we have
 \begin{equation*}
\Vert e^{-t\,\vert \DD \vert^{\alpha}}\Delta_{j}u \Vert_{L^{p}}\le C e^{-c t\,2^{j \alpha}}\Vert \Delta_{j}u \Vert_{L^{p}}.
\end{equation*}
Where the constants C and c depend only on the dimension d.
\end{lem}
Combining this estimate with Duhamel formula yields for every $j \in\NN$
\begin{equation*}
\Vert e^{-(t-\tau) \vert \DD \vert^{\alpha}} \Delta_{j}\Bar{h}_{q}(\tau) \Vert_{L^{\infty}} 
\le C e^{-c(t-\tau) 2^{j \alpha}}\bigg(\Vert S_{q-1}v-v\big)\cdot\nabla \theta_{q}\Vert_{L^{\infty}}+\Vert [\Delta_{q},v\cdot\nabla]\theta \Vert_{L^{\infty}}\bigg).
\end{equation*}\\
This implies that 
\begin{equation*}
\Vert e^{-(t-\tau) \vert \DD \vert^{\alpha}}\Delta_{j}\Bar{h}_{q}(\tau) \Vert_{L^{\infty}}\lesssim e^{-c(t-\tau) 2^{j \alpha}}\Vert \theta^{0}\Vert_{L^{\infty}}\bigg(\Vert \nabla \Delta_{-1}v \Vert_{L^{\infty}}+(q+2) \Vert \omega \Vert_{L^{\infty}}\bigg) 
\end{equation*}
and
\begin{equation*}
\Vert e^{-(t-\tau) \vert \DD \vert^{\alpha}}\Delta_{j}h^{1}_{q}(\tau) \Vert_{L^{\infty}}\lesssim e^{-c(t-\tau) 2^{j \alpha}}\Vert h_{q}^{1}(\tau) \Vert_{L^{\infty}}\;\lesssim  e^{-c(t-\tau) 2^{j \alpha}}e^{C V_{q}(t)}V_{q}(t) 2^{q \alpha}\Vert \theta_{q}\Vert_{L^{\infty}}.
\end{equation*}\\ Therefore
\begin{eqnarray*}\label{d8}
\Vert \Delta_{j}\Bar{\theta}_{q}(t)\Vert_{L^{\infty}}&\lesssim&  e^{-ct2^{j\alpha}}\Vert \Delta_{j}\theta_{q}^{0}\Vert_{L^{\infty}}+2^{q{\alpha}}e^{CV_{q}(t)}V_{q}(t)\int_{0}^{t}e^{-c(t-\tau)2^{j\alpha}}\Vert \theta_{q}(\tau) \Vert_{L^{\infty}}d\tau\\
&+&(q+2) \Vert \theta^{0}\Vert_{L^{\infty}}\int_{0}^{t}e^{-c(t-\tau)2^{j\alpha}}\Vert \omega(\tau) \Vert_{L^{\infty}}d\tau\\
&+&\Vert \theta^{0}\Vert_{L^{\infty}}\int_{0}^{t}e^{-c(t-\tau)2^{j\alpha}}\Vert \nabla \Delta_{-1}v(\tau) \Vert_{L^{\infty}}d\tau.
\end{eqnarray*}
Integrating in time and using Young inequality, we get for all $j\in\NN$
\begin{eqnarray}\label{T8}
\nonumber\Vert\Delta_{j}\Bar{\theta}_{q}\Vert_{L^1_tL^{\infty}}&\lesssim& 2^{-j\alpha}(\big\Vert\Delta_{j}\theta_{q}^{0}\Vert_{L^{\infty}}+(q+2) \Vert \theta^{0}\Vert_{L^{\infty}}\Vert \omega \Vert_{L^{1}_{t}L^{\infty}}\\
&+&\Vert \theta^{0} \Vert_{L^{\infty}}\Vert \nabla \Delta_{-1}v \Vert_{L^{1}_{t}L^{\infty}}\big)
+2^{(q-j)\alpha}e^{CV_{q}(t)}V_{q}(t) \Vert\theta_{q} \Vert_{L^{1}_{t}L^{\infty}}.
\end{eqnarray}
Let now $N\in\NN$ be a fixed number that will be chosen later. Since the flow $\psi_q$ preserves Lebesgue measure then we write
\begin{eqnarray*}
2^{ q \alpha}\Vert \theta_{q} \Vert_{L_{t}^{1}L^{\infty}}=2^{q \alpha} \Vert\Bar{\theta}_{q}\Vert_{L_{t}^{1}L^{\infty}}
&\le& 2^{q \alpha}\bigg(\displaystyle \sum_{\vert j-q \vert < N}\big\Vert \Delta_{j}\Bar{\theta}_{q}\big\Vert_{L_{t}^{1}L^{\infty}}+\sum_{\vert j-q \vert \geqslant N}\big\Vert \Delta_{j}\Bar{\theta}_{q}\big\Vert_{L_{t}^{1}L^{\infty}}\bigg)\\
&:=&\textrm{I}+\textrm{II}. 
\end{eqnarray*}
If $q \geqslant N,$ then it follows from \eqref{T8},
\begin{eqnarray*}\label{d12}
\textrm{I} &\lesssim& \displaystyle \sum_{\vert j-q \vert < N}2^{(q-j) \alpha}\bigg(\Vert \Delta_{j}\theta^{0}_{q}\Vert_{L^{\infty}}+\Vert \theta^{0} \Vert_{L^{\infty}}\left((q+2) \Vert \omega \Vert_{L^{1}_{t}L^{\infty}}+\Vert \nabla \Delta_{-1}v \Vert_{L^{1}_{t}L^{\infty}}\right)\bigg)\\
&+& \displaystyle \sum_{\vert j-q \vert < N}2^{(q-j) \alpha}V_{q}(t) e^{C V_{q}(t)}2^{q \alpha}\Vert \theta_{q}\Vert_{L^{1}_{t}L^{\infty}}.
\end{eqnarray*}
Thus\\
$\textrm{I} \lesssim \Vert \theta^{0}\Vert_{L^{\infty}}+2^{N \alpha}\Vert \theta^{0}\Vert_{L^{\infty}}\left((q+2) \Vert \omega \Vert_{L^{1}_{t}L^{\infty}}+\Vert \nabla \Delta_{-1}v \Vert_{L^{1}_{t}L^{\infty}}\right) +V_{q}(t) e^{C V_{q}(t)}2^{q \alpha}2^{N \alpha}\Vert \theta_{q} \Vert_{L^{1}_{t}L^{\infty}}.$\\\\
To estimate $\textrm{II},$ we use the following result due to Vishik \cite {vis}\\
\begin{equation*}
\big\Vert\Delta_{j}\Bar{\theta}_{q}\big\Vert_{L^{p}}\lesssim 2^{-\vert q-j\vert}e^{CV_{q}(t)}\Vert \theta_{q}\Vert_{L^{p}}\;,\forall\;p \in [1,\infty].
\end{equation*}
 Thus
\begin{eqnarray*}\label{d13}
\textrm{II}&=&2^{q \alpha}\sum_{\vert j-q \vert \geqslant N}\Vert \Delta_{j}\Bar{\theta}_{q} \Vert_{L^{1}_{t}L^{\infty}}\\
&\lesssim& 2^{q \alpha}\displaystyle \sum_{\vert j-q \vert \geqslant N}2^{-\vert q-j \vert}e^{CV_{q}(t)}\Vert \theta_{q}\Vert_{L^{1}_{t}L^{\infty}}\\
&\lesssim& 2^{-N}e^{CV_{q}(t)}2^{q \alpha}\Vert \theta_{q}\Vert_{L_{t}^{1}L^{\infty}}.
\end{eqnarray*}
We have then
\begin{eqnarray*}\label{d14}
2^{q \alpha}\Vert \theta_{q}\Vert_{L^{1}_{t}L^{\infty}}&\lesssim& \Vert \theta^{0}\Vert_{L^{\infty}}+2^{N \alpha}\Vert \theta^{0}\Vert_{L^{\infty}}\left((q+2) \Vert \omega \Vert_{L^{1}_{t}L^{\infty}}+\Vert \nabla \Delta_{-1}v \Vert_{L^{1}_{t}L^{\infty}}\right)\\ 
&+& V_{q}(t) e^{C V_{q}(t)}2^{q \alpha}2^{N \alpha}\Vert \theta_{q} \Vert_{L^{1}_{t}L^{\infty}}+2^{-N}e^{CV_{q}(t)}2^{q \alpha}\Vert \theta_{q}\Vert_{L_{t}^{1}L^{\infty}}.
\end{eqnarray*}
For low frequencies $q<N,$ we have 
$$2^{q \alpha}\Vert \theta_{q}\Vert_{L^{1}_{t}L^{\infty}}\lesssim 2^{N \alpha}\Vert \theta \Vert_{L^{1}_{t}L^{\infty}}.$$
Therefore we get for $q \geqslant-1$
\begin{eqnarray*}\label{d15}
2^{q \alpha}\Vert \theta_{q}\Vert_{L^{1}_{t}L^{\infty}}&\le& C \Vert \theta^{0}\Vert_{L^{\infty}}+C 2^{N \alpha}\Vert \theta \Vert_{L^{1}_{t}L^{\infty}}+C 2^{N \alpha}\Vert \theta^{0}\Vert_{L^{\infty}}\left((q+2) \Vert \omega \Vert_{L^{1}_{t}L^{\infty}}+\Vert \nabla \Delta_{-1}v \Vert_{L^{1}_{t}L^{\infty}}\right)\\ 
&+& C \left( V_{q}(t) e^{C V_{q}(t)}2^{N \alpha}+2^{-N}e^{CV_{q}(t)} \right) 2^{q \alpha}\Vert \theta_{q}\Vert_{L_{t}^{1}L^{\infty}}.
\end{eqnarray*}
Now, we claim that there exists two absolute constants $N\in\NN$ and $C_{1}>0$ such that if $V_{q}(t) \le C_{1},$ then
\begin{equation*}
V_{q}(t) e^{CV_{q}(t)}2^{N \alpha }+2^{-N}e^{C V_{q}(t)}\le \frac{1}{2C}\cdot
\end{equation*} 
To show this, we take first $t$ such that $V_{q}(t) \le 1,$ which is possible since $\displaystyle \lim_{t\rightarrow 0^{+}}V_{q}(t)=0$. Second, we choose $N$ in order to get $2^{-N}e^{C}\le \frac{1}{4C}$. By taking again $V_{q}(t)$ sufficiently small we obtain that $V_{q}(t) e^{C V_{q}(t)}2^{N \alpha }\le \frac{1}{4 C}.$\\
Under this assumption $V_{q}(t) \le C_{1},$ we obtain for $q \geqslant -1$ 
\begin{equation}\label{T9}
2^{q \alpha}\Vert \theta_{q}\Vert_{L^{1}_{t}L^{\infty}}\lesssim \Vert \theta \Vert_{L^{1}_{t}L^{\infty}}+\Vert \theta^{0}\Vert_{L^{\infty}}\bigg(1+(q+2) \Vert \omega \Vert_{L^{1}_{t}L^{\infty}}+\Vert \nabla \Delta_{-1}v \Vert_{L^{1}_{t}L^{\infty}}\bigg). 
\end{equation}
We use H\"older's inequality and Lemma \ref{lem3} for $\Vert \theta \Vert_{L^{1}_{t}L^{\infty}},$ we get $\Vert \theta \Vert_{L^{1}_{t}L^{\infty}}\le t \Vert \theta^{0} \Vert_{L^{\infty}}.$ \\ Hence
\begin{equation*}
2^{q \alpha}\Vert \theta_{q}\Vert_{L^{1}_{t}L^{\infty}}\lesssim \Vert \theta^{0}\Vert_{L^{\infty}}\bigg(1+t+(q+2) \Vert \omega \Vert_{L^{1}_{t}L^{\infty}}+\Vert \nabla \Delta_{-1}v \Vert_{L^{1}_{t}L^{\infty}}\bigg). 
\end{equation*}
Therefore the result is proved for small time. 
\subsection{Globalization}
Let us now see how to extend this for arbitrary large time $T.$ We take a partition $(T_{i})^{M}_{i=0}$ of $[0,T]$ i.e $0=T_{0}\le T_{1}\le ...\le T_{M}=T$ and such that 
$$\displaystyle \int^{T_{i+1}}_{T_{i}}\Vert \nabla S_{q-1}v(t) \Vert_{L^{\infty}} dt \thickapprox C_{1}\;,\;\forall\;i \in[0,M].$$
Reproducing the same argument in \eqref{T9} we find in view of $\Vert\theta(T_{i}) \Vert_{L^{\infty}}\le \Vert \theta^{0}\Vert_{L^{\infty}}.$
\begin{eqnarray*}\label{xx}
2^{q \alpha}\displaystyle \int^{T_{i+1}}_{T_{i}}\Vert \theta_{q}(t)\Vert_{L^{\infty}}dt &\lesssim& \displaystyle \int^{T_{i+1}}_{T_{i}}\Vert \theta(t) \Vert_{L^{\infty}}dt+\Vert \theta^{0}\Vert_{L^{\infty}}\\
&+&\Vert \theta^{0}\Vert_{L^{\infty}}(q+2) \displaystyle \int^{T_{i+1}}_{T_{i}}\Vert \omega(t) \Vert_{L^{\infty}}dt\\
&+&\Vert \theta^{0}\Vert_{L^{\infty}}\displaystyle \int^{T_{i+1}}_{T_{i}}\Vert \nabla \Delta_{-1}v(t) \Vert_{L^{\infty}}dt.
\end{eqnarray*}
Summing these estimates of 0 to M, we get 
\begin{eqnarray*}\label{d17}
2^{q \alpha}\Vert \theta_{q}\Vert_{L^{1}_{T}L^{\infty}}&\lesssim& \Vert \theta \Vert_{L^{1}_{T}L^{\infty}}+(M+1) \Vert \theta^{0}\Vert_{L^{\infty}}\\
&+&\Vert \theta^{0}\Vert_{L^{\infty}}\bigg((q+2) \Vert \omega \Vert_{L^{1}_{T}L^{\infty}}+\Vert \nabla \Delta_{-1}v \Vert_{L^{1}_{T}L^{\infty}}\bigg). 
\end{eqnarray*}
As $M \thickapprox V_{q}(t),$ then 
\begin{eqnarray*}\label{d18}
2^{q \alpha}\Vert \theta_{q}\Vert_{L^{1}_{T}L^{\infty}}&\lesssim& \Vert \theta \Vert_{L^{1}_{T}L^{\infty}}+\big(V_{q}(t)+1\big) \Vert \theta^{0}\Vert_{L^{\infty}}\\
&+& \Vert \theta^{0}\Vert_{L^{\infty}}\bigg((q+2) \Vert \omega \Vert_{L^{1}_{T}L^{\infty}}+\Vert \nabla \Delta_{-1}v \Vert_{L^{1}_{T}L^{\infty}}\bigg). 
\end{eqnarray*}
From the definition of the operator $S_{q-1},$ we have then
\begin{eqnarray*}\label{d19}
S_{q-1}v&=&\sum_{-1 \le p \le q-2}\Delta_{p}v\\
&=&\Delta_{-1}v+\displaystyle \sum^{q-2}_{p=0}\Delta_{p}v.
\end{eqnarray*} 
Thus
\begin{eqnarray*}\label{d20}
\Vert \nabla S_{q-1}v \Vert_{L^{\infty}}&\le& \Vert \nabla \Delta_{-1}v \Vert_{L^{\infty}}+\displaystyle \sum^{q-2}_{p=0}\Vert \nabla \Delta_{p}v \Vert_{L^{\infty}}\\
&\le& \Vert \nabla \Delta_{-1}v \Vert_{L^{\infty}}+C \displaystyle \sum^{q-2}_{p=0}\Vert \Delta_{p}\omega \Vert_{L^{\infty}}\\
&\le& \Vert \nabla \Delta_{-1}v \Vert_{L^{\infty}}+C (q-1) \Vert \omega \Vert_{L^{\infty}}.
\end{eqnarray*}

 We have used Bernstein inequality and the classical result 
$$\Vert \Delta_{q}v \Vert_{L^{p}}\thickapprox 2^{-q}\Vert \Delta_{q}\omega \Vert _{L^{p}}\;,\;\forall\;p \in[1,\infty].$$ 
Therefore
$$\Vert \nabla S_{q-1}v \Vert_{L^{\infty}}\lesssim \Vert \nabla \Delta_{-1}v \Vert_{L^{\infty}}+ (q+2) \Vert \omega \Vert_{L^{\infty}},$$
then inserting this estimate into the previous one
\begin{equation*}
2^{q \alpha}\Vert \theta_{q}\Vert_{L^{1}_{T}L^{\infty}}\lesssim \Vert \theta^{0}\Vert_{L^{\infty}}\bigg(1+T+(q+2) \Vert \omega \Vert_{L^{1}_{T}L^{\infty}}+\Vert \nabla \Delta_{-1}v \Vert_{L^{1}_{T}L^{\infty}}\bigg). 
\end{equation*}
This is the desired result.

\end{proof}
\begin{Rema}\label{R}    
If the velocity belongs to $L^1_t{\rm Lip}$ then the previous  estimate becomes
$$\forall\;1 \le \rho \le \infty\;,\;2^{q {\frac{\alpha}{\rho}} }\Vert\Delta_{q}\theta\Vert_{L^\rho_tL^\infty}\lesssim\Vert \theta^0\Vert_{L^\infty}\bigg(1+t+\|\nabla v\|_{L^1_t L^\infty}\bigg).$$
\end{Rema}

\section{Proof of Theorem \ref{TX}}\label{W2}
This section is devoted to the proof of theorem \ref{TX}. For conciseness, we shall provide the a priori estimates supporting the claim of the theorem and give a complete proof of the uniqueness, while the proof of the existence part will be shortened and briefly described since it is somewhat contained in \cite{hk1}.

\subsection{a Priori Estimates} 
The important quantities to bound for all times are the $L^\infty$ norm of the vorticity and the Lipschitz norm of the velocity. 
The main step for obtain a Lipschitz bound is give an $L^\infty$-bound of the vorticity. We prove before an $L^p$ estimate with $p<\infty$ and this allows us to bound the vorticity in $L^\infty.$ We start then with the following one,
\begin{prop}\label{prop7}
Let $(\alpha, p)\in]1, 2]\times[1,\infty[\;,\;\omega^0\in L^\infty\cap L^p$, $\theta^0\in B^{1-\alpha}_{p,1}\cap L^{\infty},$ and $F\in\mathcal{C}^{1}(\RR,\RR).$ Then any smooth  solution of the system \eqref{T1} satisfies\\
$\bullet$\qquad\qquad\qquad$\Vert\theta(t)\Vert_{L^{\infty}}\le \Vert \theta^{0}\Vert_{L^{\infty}}.$\\
$\bullet$\qquad\qquad\qquad$\|\omega(t)\|_{L^\infty\cap L^{p}}+\|\nabla\theta\|_{L^1_tL^\infty}\le C_0 e^{C_0 t}.$
\end{prop}
\begin{proof}
The first inequality is a direct consequence of Lemma \ref{lem3}. For the second one, we start with the vorticity equation $$\partial_{t}\omega+v\cdot\nabla\omega=F^{\prime}_{2}(\theta)\partial_{1}\theta- F^{\prime}_{1}(\theta)\partial_{2}\theta.$$
Taking the $L^{p}$ norm we get
\begin{equation*}
\Vert\omega(t)\Vert_{L^{p}}\le\Vert\omega^{0}\Vert_{L^{p}}+\int^{t}_{0}\bigg( \sum^{2}_{i=1} \Vert F^{\prime}_{i}(\theta) \Vert_{L^{\infty}}\bigg) \Vert \nabla \theta(\tau) \Vert_{L^{p}}\,d\tau.
\end{equation*}
Since $\Vert F^{\prime}_{i}(\theta) \Vert_{L^{\infty}}= \displaystyle \sup_{x,t}\vert F^{\prime}_{i} (\theta (x,t) )\vert$, then we have
\begin{eqnarray*}\label{d21}
\Vert F^{\prime}_{i}(\theta) \Vert_{L^{\infty}}&\le& \sup_{\vert x \vert \leqslant \Vert \theta \Vert_{L^{\infty}}}\vert F^{\prime}_{i} (x) \vert\\
&\le&\displaystyle \sup_{\vert x \vert \leqslant \Vert \theta^{0} \Vert_{L^{\infty}}}\vert F^{\prime}_{i} (x)\vert.
\end{eqnarray*}\\
As $F \in \mathcal{C}^{1}(\RR,\RR)$ and $\theta^{0} \in L^{\infty},$ then
$$\Vert F^{\prime}_{i}(\theta)\Vert_{L^{\infty}}\le \sup_{\vert x \vert \le \Vert\theta^{0} \Vert_{L^{\infty}}}\Vert\nabla F_{i}\Vert_{L^{\infty}}\leqslant C.$$
Thus
\begin{equation*}
\Vert \omega(t) \Vert_{L^{p}}\lesssim \Vert \omega^{0}\Vert_{L^{p}}+\Vert \nabla \theta \Vert_{L^{1}_{t}L^{p}}.
\end{equation*} 
To estimate $\Vert \nabla \theta \Vert_{L^{1}_{t}L^{p}},$ we use Bernstein inequality and Proposition \ref{prop5}, we get 
\begin{eqnarray*}\label{d22}
\Vert \nabla \theta \Vert_{L^{1}_{t}L^{p}}&\le& \sum_{q \geqslant 0}\Vert \nabla \Delta_{q}\theta \Vert_{L^{1}_{t}L^{p}}+ \Vert \nabla \Delta_{-1}\theta \Vert_{L^{1}_{t}L^{p}}\\
&\lesssim& \sum_{q \geqslant 0}2^{q}\Vert   \Delta_{q}\theta \Vert_{L^{1}_{t}L^{p}} + \Vert \Delta_{-1}\theta \Vert_{L^{1}_{t}L^{p}}\\
&\lesssim& \displaystyle \sum_{q \geqslant0}2^{q(1- \alpha)}\Vert   \Delta_{q}\theta^{0}\Vert_{L^{p}} +\displaystyle \sum_{q \geqslant 0}2^{q(1- \alpha)}\Vert \theta^{0}\Vert_{L^{\infty}}\Vert \omega \Vert_{L^{1}_{t}L^{p}}\\
&+& t \big( \Vert \Delta_{-1}\theta^{0} \Vert_{L^{p}}+\Vert \theta^{0}\Vert_{L^{\infty}} \Vert \omega\Vert_{L^{1}_{t}L^{p}}\big)\\
&\lesssim& \Vert \theta^{0}\Vert_{B^{1-\alpha}_{p,1}}+\Vert \theta^{0}\Vert_{L^{\infty}}\Vert \omega \Vert_{L^{1}_{t}L^{p}}\\
&\lesssim& \Vert \theta^{0}\Vert_{B^{1-\alpha}_{p,1}\cap L^{\infty}}(1+\Vert \omega \Vert_{L^{1}_{t}L^{p}}).
\end{eqnarray*}
Therefore 
\begin{equation*}
\Vert \omega(t) \Vert_{L^{p}}\lesssim \Vert \omega^{0}\Vert_{L^{p}}+ \Vert \theta^{0}\Vert_{B^{1-\alpha}_{p,1}\cap L^{\infty}}(1+\int^{t}_{0}\Vert \omega(\tau) \Vert_{L^{p}} d\tau).
\end{equation*}
Using Gronwall's inequality, we obtain then
\begin{equation}\label{T10}
\Vert \omega(t) \Vert_{L^{p}}\lesssim C_{0}e^{C_{0}t}.
\end{equation}
Where $C_{0}$ is a constant depending on the initial data. We can now estimate $\Vert \omega(t) \Vert_{L^{\infty}}.$ For this we take the $L^{\infty}$ norm of the velocity equation, we get 
\begin{equation*}
\Vert\omega(t)\Vert_{L^{\infty}}\le\Vert\omega^{0}\Vert_{L^{\infty}}+\int^{t}_{0}
\bigg(\sum^{2}_{i=1} \Vert F^{\prime}_{i}(\theta) \Vert_{L^{\infty}}\bigg) \Vert \nabla \theta(\tau) \Vert_{L^{\infty}}\,d\tau.
\end{equation*}
Then we have 
\begin{equation*}
\Vert\omega(t)\Vert_{L^{\infty}}\lesssim \Vert \omega^{0}\Vert_{L^{\infty}}+\Vert \nabla \theta \Vert_{L_{t}^{1}L^{\infty}}.
\end{equation*} 
Using the classical embedding  $B_{\infty, 1}^{1}\hookrightarrow \textnormal{Lip}(\RR^2)$, we obtain
\begin{equation*}
\Vert \omega(t) \Vert_{L^{\infty}}\lesssim \Vert \omega^{0}\Vert_{L^{\infty}}+\Vert
\theta \Vert_{L_{t}^{1}B_{\infty, 1}^{1}}.
\end{equation*}
We have from the definition of Besov space and H\"older inequality,
\begin{eqnarray*}\label{d23}
\Vert \theta \Vert_{L^{1}_{t}B^{1}_{\infty,1}}&=&\sum_{q \geqslant -1}2^{q}\Vert \Delta_{q}\theta \Vert_{L^{1}_{t}L^{\infty}}\\
&=& C \Vert \Delta_{-1}\theta \Vert_{L^{1}_{t}L^{\infty}}+\displaystyle \sum_{q \in \NN}2^{q}\Vert \Delta_{q}\theta \Vert_{L^{1}_{t}L^{\infty}}
\end{eqnarray*}
Then we have
$$\Vert \theta \Vert_{L^{1}_{t}B^{1}_{\infty,1}}\le C t \Vert \theta^{0}\Vert_{L^{\infty}}+\displaystyle \sum_{q \in \NN}2^{q}\Vert \Delta_{q}\theta \Vert_{L^{1}_{t}L^{\infty}}.$$
We use now Proposition \ref{prop6}, we obtain then
\begin{equation*}
\sum_{q \in \NN}2^{q}\Vert \Delta_{q}\theta \Vert_{L^{1}_{t}L^{\infty}}\le \sum_{q \in \NN}2^{q(1-\alpha)}\Vert \theta^{0}\Vert_{L^{\infty}}\bigg(1+t+(q+2) \Vert \omega \Vert_{L^{1}_{t}L^{\infty}}+\Vert \nabla \Delta_{-1}v \Vert_{L^{1}_{t}L^{\infty}}\bigg).
\end{equation*}
Since $\alpha  >1,$ then the series $\displaystyle \sum_{q \in \NN}2^{q(1-\alpha)}$ and $\displaystyle \sum_{q \in \NN}2^{q(1-\alpha)}(q+2)$ are convergent. It follows that 
\begin{equation*}
\sum_{q \in \NN}2^{q}\Vert \Delta_{q}\theta \Vert_{L^{1}_{t}L^{\infty}}\lesssim \Vert \theta^{0}\Vert_{L^{\infty}}\left(1+t+\Vert \omega\Vert_{L^{1}_{t}L^{\infty}}+\Vert \nabla \Delta_{-1}v \Vert_{L^{1}_{t}L^{\infty}}\right) .
\end{equation*}
This implies that
\begin{eqnarray*}\label{d24}
\Vert \theta \Vert_{L^{1}_{t}B^{1}_{\infty,1}}&\lesssim& \Vert \theta^{0}\Vert_{L^{\infty}}\left(1+t+\Vert \omega\Vert_{L^{1}_{t}L^{\infty}}+\Vert \nabla \Delta_{-1}v \Vert_{L^{1}_{t}L^{\infty}}\right) \\
&\lesssim& \Vert \theta^{0}\Vert_{L^{\infty}}\left(1+t+\Vert \omega \Vert_{L^{1}_{t}L^{\infty}}+\Vert \nabla \Delta_{-1}v \Vert_{L^{1}_{t}L^{p}}\right) \\
&\lesssim& \Vert \theta^{0}\Vert_{L^{\infty}}\left( 1+t+\Vert \omega \Vert_{L^{1}_{t}L^{\infty}}+\Vert \nabla v \Vert_{L^{1}_{t}L^{p}}\right)\\
&\lesssim& \Vert \theta^{0}\Vert_{L^{\infty}}\left(1+t+\Vert \omega \Vert_{L^{1}_{t}L^{\infty}}+\Vert \omega \Vert_{L^{1}_{t}L^{p}}\right).
\end{eqnarray*}
Therefore
\begin{equation}\label{T11}
\Vert \theta \Vert_{L^{1}_{t}B^{1}_{\infty,1}}\lesssim \Vert \theta^{0}\Vert_{L^{\infty}}\left(1+t+\Vert \omega \Vert_{L^{1}_{t}(L^{\infty}\cap L^{p})}\right) .
\end{equation}
We have used the classical result
$\Vert \nabla v \Vert_{L^{p}}\thickapprox \Vert \omega \Vert_{L^{p}}\;,\;\forall p \in]1,\infty[.$ Hence from \eqref{T10} and \eqref{T11}, we get
\begin{eqnarray*}\label{d25}
\Vert \omega(t) \Vert_{L^{\infty}}&\lesssim& \Vert \omega^{0}\Vert_{L^{\infty}}+\Vert \theta^{0}\Vert_{L^{\infty}}\left( 1+t+\int^{t}_{0}C_{0}e^{C_{0}\tau}d\tau+\int^{t}_{0}\Vert \omega(\tau) \Vert_{L^{\infty}}d\tau\right)\\
&\lesssim& \Vert \omega^{0}\Vert_{L^{\infty}}+\Vert \theta^{0}\Vert_{L^{\infty}}\left( 1+t+C_{0}e^{C_{0}t}\right)+\Vert \theta^{0}\Vert_{L^\infty}\displaystyle \int^{t}_{0}\Vert \omega(\tau) \Vert_{L^{\infty}} d\tau.
\end{eqnarray*}
According to Gronwall's inequality, one has
\begin{equation}\label{T12}
\Vert \omega(t) \Vert_{L^{\infty}}\lesssim C_{0}e^{C_{0}t}.
\end{equation}
Consequently
\begin{equation}\label{T13}
\Vert \omega(t) \Vert_{L^{\infty}\cap L^{p}}\lesssim C_{0}e^{C_{0}t}.
\end{equation}
 Plugging this estimate  into \eqref{T11} gives
\begin{equation}\label{T14}
\|\theta\|_{L^1_{t}B_{\infty,1}^{1}}\leq C_{0}e^{C_{0}t}.
\end{equation}
This gives in view of Besov embedding $$\|\nabla \theta\|_{L^1_t L^\infty}\le  C_{0}e^{C_{0}t}.$$
This conclude the proof of the proposition.\\
\end{proof}
We shall now give a bound for the norm Lipschitz of the velocity.
\begin{prop}\label{prop8}
Under the same assumptions of Proposition \ref{prop7} and if in addition  $\omega^0\in B_{\infty,1}^0$and $F \in \mathcal{C}^{3}(\RR,\RR),$ then we have for every $t\in\RR_+$
$$\Vert\omega(t)\Vert_{L^{\infty}_{t}B_{\infty, 1}^{0}}+\|\nabla v(t)\|_{L^\infty}\le C_{0}e^{\exp{C_{0}t}}.$$
\end{prop}
\begin{proof} We decompose $v$ in frequencies as
$$v(t)= \Delta_{-1}v(t) + \sum_{q \geqslant 0} \Delta_{q}v(t)$$
then we have
\begin{eqnarray*}\label{d26}
\| \nabla v(t) \|_{L^\infty}&\le&\| \nabla \Delta_{-1}v(t) \|_{L^\infty}+\displaystyle \sum_{q \geqslant 0}\| \nabla \Delta_{q} v(t) \|_{L^\infty}\\
&\lesssim& \| \nabla \Delta_{-1}v(t) \|_{L^{p}}+\displaystyle \sum_{q \geqslant 0}2^{q}\Vert \Delta_{q}v(t) \Vert_{L^{\infty}}\\
&\lesssim& \Vert \nabla v(t) \Vert_{L^{p}}+\displaystyle \sum_{q \geqslant 0}\Vert \Delta_{q}\omega(t) \Vert_{L^{\infty}}.
\end{eqnarray*}
Hence 
\begin{equation}\label{T15}
\| \nabla v(t) \|_{L^\infty}\lesssim \| \omega(t) \|_{L^{p}}+\| \omega(t) \|_{\widetilde L^\infty_{t}B_{\infty,1}^0}.    
\end{equation}
Let us now turn to the estimate of $\|\omega\|_{\widetilde L^\infty_{t}B_{\infty,1}^0}.$ We apply Proposition \ref{prop1} to the vorticity equation,
\begin{eqnarray*}
\|\omega\|_{\widetilde L^\infty_{t}B_{\infty,1}^0}\lesssim\left(\|\omega^0\|_{B_{\infty,1}^0}+\int^{t}_{0}(\Vert \partial_{1}(F_{2}(\theta(\tau))) \Vert_{B^{0}_{\infty,1}}+\Vert \partial_{2}(F_{1}(\theta(\tau))) \Vert_{B^{0}_{\infty,1}})d\tau\right)\left(1+\|\nabla v\|_{L^1_{t}L^\infty}\right).
\end{eqnarray*}
By the definition of Besov spaces and Lemma \ref{lem2}, we find :
\begin{equation*}
\Vert \partial_{1}(F_{2}(\theta)) \Vert_{B^{0}_{\infty,1}}
\leqslant C\Vert F_{2}(\theta) \Vert_{B^{1}_{\infty,1}}
\end{equation*}
To estimate $\Vert F_{2}(\theta) \Vert_{B^{1}_{\infty,1}}$ we use Theorem \ref{theo2},
\begin{eqnarray*}\label{d27}
\Vert F_{2}(\theta) \Vert_{B^{1}_{\infty,1}}&\le& C \sup_{\vert x \vert \le C \Vert \theta \Vert_{L^{\infty}}}\Vert F_{2}^{[1]+2}(x) \Vert_{L^{\infty}}\Vert \theta \Vert_{B^{1}_{\infty,1}}\\
&\le& C \displaystyle \sup_{\vert x \vert \le C \Vert \theta^{0} \Vert_{L^{\infty}}}\Vert F_{2}^{(3)}(x) \Vert_{L^{\infty}}\Vert \theta \Vert_{B^{1}_{\infty,1}}\\
&\le& C \Vert \theta \Vert_{B^{1}_{\infty,1}}.
\end{eqnarray*} 
Therefore
$$\Vert \partial_{1}(F_{2}(\theta)) \Vert_{B^{0}_{\infty,1}}\lesssim \Vert \theta \Vert_{B^{1}_{\infty,1}}.$$
Similarly, we obtain $\Vert \partial_{2}(F_{1}(\theta)) \Vert_{B^{0}_{\infty,1}}\lesssim \Vert \theta \Vert_{B^{1}_{\infty,1}}.$ Finally we get 
 \begin{equation}\label{T16}
\|\omega\|_{\widetilde L^\infty_{t}B_{\infty,1}^0}\lesssim\big(\|\omega^0\|_{B_{\infty,1}^0}+\|\theta\|_{L^1_{t}B_{\infty,1}^1}\big)\big(1+\|\nabla v\|_{L^1_{t}L^\infty}\big)
\end{equation}
Putting together \eqref{T10},\eqref{T14},\eqref{T15} and \eqref{T16} and using Gronwall's inequality, we deduce
\begin{equation}\label{T17}
\|\omega\|_{\widetilde L^\infty_{t}B_{\infty,1}^0}+\|\nabla v(t)\|_{L^\infty}\le C_0 e^{\exp{C_0 t}}.
\end{equation}
\end{proof}
Now we will describe the last part of the a priori estimates.
\begin{prop}\label{prop9} Let $(\alpha,p) \in]1,2] \times]1, \infty[\;,\;v^{0}\in B_{p, 1}^{1+\frac{2}{p}}$ be a divergence free vector-field of $\RR^2\;,\,\theta^{0}\in B_{p, 1}^{-\alpha+1+\frac{2}{p}}\cap L^{\infty}$ and $F \in \mathcal{C}^{5}(\RR,\RR).$ Then for every $\rho \geqslant 1$ and for every $t \in \RR_{+}$, 
\begin{equation*}
\Vert \theta \Vert_{\widetilde L^{\rho}_{t}B^{\frac{\alpha}{\rho}}_{\infty,\infty}} +\Vert\theta\Vert_{L_{t}^{1}B_{p, 1}^{1+\frac{2}{p}}}+\Vert v \Vert_{{L}_{t}^{\infty}B_{p, 1}^{1+\frac{2}{p}}}\le C_0 e^{e^{\exp{C_0 t}}}.
\end{equation*}
\end{prop}
\begin{proof}
For the first estimate $\Vert\theta\Vert_{\widetilde {L}^{\rho}_{t}B_{\infty,\infty}^{\frac{\alpha}{\rho}}},$ it suffices to combine Remark \ref{R} with Lipschitz estimate of the velocity \eqref{T17} as follows
\begin{eqnarray*}\label{d28}
2^{q \frac{\alpha}{\rho}}\Vert \Delta_{q}\theta \Vert_{L^\rho_{t}L^{\infty}}&\lesssim& \Vert \theta^{0}\Vert_{L^{\infty}}\bigg(1+t+\Vert \nabla v \Vert_{L^{1}_{t}L^{\infty}}\bigg)\\
&\lesssim& C_{0}e^{\exp C_{0}t}.
\end{eqnarray*}
Hence, it follows that 
$$\| \theta \|_{\widetilde {L}^{\rho}_{t} B^{\frac{\alpha}{\rho}}_{\infty,\infty}}\le C_{0} e^{\exp{C_{0} t}}.$$
In order to prove  the second estimate of the Proposition, we need to split the proof in  two cases : $s<1$ and $s\geqslant1$ with $s=-\alpha+1+\frac2p.$\\
$\bullet$ \textbf{First case} $s=-\alpha+1+\frac2p<1$. We apply Proposition \ref{prop4} to the temperature equation, we get
\begin{equation*}
\|\theta\|_{L^1_tB_{p,1}^{1+\frac{2}{p}}}\lesssim\|\theta^0\|_{B_{p,1}^{-\alpha+1+\frac{2}{p}}}(1+t)e^{CV(t)}.
\end{equation*}
Since
\begin{equation*}
V(t)= \int^{t}_{0}\Vert \nabla v(\tau) \Vert_{L^{\infty}}\,d\tau\\
\leqslant \int^{t}_{0}C_0 e^{\exp{C_{0}\tau}}\,d\tau \leqslant C_0 e^{\exp{C_{0}t}}.
\end{equation*}
Finally, we obtain for the first case
\begin{equation*}
\|\theta\|_{L^1_tB_{p,1}^{1+\frac{2}{p}}}\lesssim C_{0}e^{e^{\exp C_{0}t}}.
\end{equation*}
$\bullet$ \textbf{Second case} $s=-\alpha+1+\frac2p\geq 1$. Applying once again Proposition \ref{prop4} we get by H\"older's inequality 
\begin{equation*}
\|\theta\|_{L^1_tB_{p,1}^{1+\frac{2}{p}}}\lesssim\|\theta^0\|_{B_{p,1}^{-\alpha+1+\frac{2}{p}}}\Big(1+t+\|\nabla\theta\|_{L^1_tL^\infty}\|v\|_{L^\infty_t B_{p,1}^{-\alpha+1+\frac{2}{p}}}\Big)e^{CV(t)}.
\end{equation*}
Using Proposition \ref{prop7} and \eqref{T17}, we obtain easily
\begin{equation*}
\|\theta\|_{L^1_tB_{p,1}^{1+\frac{2}{p}}}\le C_0 e^{e^{\exp C_0 t}}\,\big(1+\|v\|_{L^\infty_t B_{p,1}^{-\alpha+1+\frac{2}{p}}}\big).
\end{equation*}
Now, applying Proposition \ref{prop3} to the velocity equation we have 
\begin{eqnarray*}
\|v\|_{L^\infty_t B_{p,1}^{-\alpha+1+\frac{2}{p}}}&\lesssim& e^{CV(t)}\big(\|v^0\|_{ B_{p,1}^{-\alpha+1+\frac{2}{p}}}+\|F(\theta)\|_{L^1_t B_{p,1}^{-\alpha+1+\frac{2}{p}}}\big).
\end{eqnarray*}
For the $\Vert F(\theta)\Vert_{B_{p,1}^{-\alpha+1+\frac{2}{p}}},$ we use Theorem \ref{theo2}, then we have 
$$\Vert F(\theta)\Vert_{B_{p,1}^{-\alpha+1+\frac{2}{p}}}\le C \displaystyle \sup_{\vert x \vert \le C \Vert \theta^{0} \Vert_{L^{\infty}}}\Vert F^{[-\alpha+1+\frac{2}{p}]+2}(x)\Vert_{L^{\infty}}\Vert \theta \Vert_{B^{-\alpha+1+\frac{2}{p}}_{p,1}}.$$
Since $-\alpha+1+\frac{2}{p}\geqslant1,\;\;3\le[-\alpha+1+\frac{2}{p}]+2 <4$ and $F\in\mathcal{C}^{5}(\RR,\RR),$ we deduce then 
\begin{equation*}
\Vert F(\theta) \Vert_{B_{p,1}^{-\alpha+1+\frac{2}{p}}}\lesssim\Vert \theta \Vert_{B^{-\alpha+1+\frac{2}{p}}_{p,1}}.
\end{equation*}
Hence we get 
\begin{eqnarray*}
\Vert v \Vert_{L^\infty_t B_{p,1}^{-\alpha+1+\frac{2}{p}}}\lesssim C_0 e^{e^{\exp C_0 t}}\big( 1+\|\theta\|_{L^1_t B_{p,1}^{-\alpha+1+\frac{2}{p}}}\big).
\end{eqnarray*} 
Consequently 
$$\|\theta\|_{L^1_tB_{p,1}^{1+\frac{2}{p}}}\lesssim C_0 e^{e^{\exp C_0 t}}\big( 1+\|\theta\|_{L^1_t B_{p,1}^{-\alpha+1+\frac{2}{p}}}\big).$$
Iterating this procedure we get for $n\in\NN$
$$\|\theta\|_{L^1_tB_{p,1}^{1+\frac{2}{p}}}\le C_0 e^{e^{\exp C_0 t}}\big( 1+\|\theta\|_{L^1_t B_{p,1}^{-n\alpha+1+\frac{2}{p}}}\big).$$
To conclude it is enough to choose $n$ such  that $-(n+1)\alpha+1+\frac{2}{p}<1$  and then  we can apply the first case. Finally we get
\begin{equation}\label{T18}
\|\theta\|_{L^1_tB_{p,1}^{1+\frac{2}{p}}}\le C_0 e^{e^{\exp C_0 t}}.
\end{equation}
Applying again Proposition \ref{prop3} to the velocity equation, we obtain
\begin{eqnarray*}
\|v(t)\|_{B_{p,1}^{1+\frac 2p}}&\lesssim& e^{CV(t)}\big( \|v^0\|_{B_{p,1}^{1+\frac 2p}}+\|F(\theta)\|_{L^1_tB_{p,1}^{1+\frac{2}{p}}}\big).\\
\end{eqnarray*}
As above, we use Theorem \ref{theo2} for $\Vert F(\theta) \Vert_{B^{1+\frac{2}{p}}_{p,1}}$ we obtain
\begin{eqnarray*}
\Vert F(\theta) \Vert_{B_{p,1}^{1+\frac 2p}}\le C \sup_{\vert x \vert \le C\Vert \theta^{0}\Vert_{L^{\infty}}}\Vert F^{[1+\frac{2}{p}]+2}(x)\Vert_{L^{\infty}}\Vert \theta \Vert_{B_{p,1}^{1+\frac 2p}}.
\end{eqnarray*}
Here $[1+\frac{2}{p}]+2<5,\,F \in \mathcal{C}^{5}(\RR,\RR)$ and $\theta^{0} \in L^{\infty},$ then we get
\begin{eqnarray*}
\Vert F(\theta) \Vert_{L^{1}_{t}B_{p,1}^{1+\frac 2p}} \lesssim \Vert \theta \Vert_{L^{1}_{t}B_{p,1}^{1+\frac 2p}}.
\end{eqnarray*}
We obtain from \eqref{T18} 
\begin{eqnarray*}\label{d29}
\|v(t)\|_{B_{p,1}^{1+\frac 2p}}&\lesssim& C_{0}e^{e^{\exp C_{0}t}}\big( 1+\|\theta\|_{L^1_tB_{p,1}^{1+\frac{2}{p}}}\big)\\
&\lesssim&C_0 e^{e^{\exp C_0 t}}.
\end{eqnarray*}
Therefore
$$\Vert v(t) \Vert_{L^{\infty}_{t}B^{1+\frac{2}{p}}_{p,1}}\lesssim C_{0}e^{e^{\exp C_{0}t}}.$$
\end{proof}

\subsection{A uniqueness result.}
In this paragraph, we will establish a uniqueness result for the system \eqref{T1} in the following space.
$$\mathcal{A}_T= (L^{\infty}_{T}L^{p}\cap L^{1}_{T} \dot{W}^{1,p}\cap L^{1}_{T}{\rm Lip})\times (L^{1}_{T} L^{p}\cap L^{\infty}_{T}(B^{-\alpha+1+\frac2p}_{p,1}\cap L^{\infty})\cap L^{1}_{T}{\rm Lip}),\qquad 1<p<\infty.$$
We take two solutions $\left\lbrace (v_{j},\theta_{j})\right\rbrace ^{2}_{j=1}$ for \eqref{T1} belonging to the space $\mathcal{A}_{T},$ for a fixed time $T>0,$ with initial data $(v^{0}_{j}, \theta^{0}_{j}),\,j=1,2.$  
We set $$v=v_2-v_1,\quad\theta=\theta_2-\theta_1\quad {\rm and }\quad p=p_2-p_1.$$
Then we find the equations
\begin{equation*}\label{thr}
\left\{ 
\begin{array}{ll}
\partial_t v+v_2\cdot\nabla v=-\nabla p-v\cdot\nabla v_1+F(\theta_{1})-F(\theta_{2})\\ 
\partial_t\theta+v_2\cdot\nabla \theta+\vert D \vert^{\alpha}\theta=-v\cdot\nabla\theta_1\\
v_{| t=0}=v^0, \quad \theta_{| t=0}=\theta^0. 
\end{array} \right.
\end{equation*}
Taking the $L^{p}$ norm of the velocity, we get
$$\Vert v(t) \Vert_{L^{p}}\leqslant \Vert v^{0}\Vert_{L^{p}}+\int^{t}_{0}\big(\Vert v(\tau) \Vert_{L^{p}} \Vert \nabla v_{1}(\tau) \Vert_{L^{\infty}}+ \Vert \nabla p(\tau) \Vert_{L^{p}}\big)\,d\tau +\Vert F(\theta_{1})-F(\theta_{2}) \Vert_{L^{1}_{t}L^{p}}.$$
To estimate the pressure, we write the following identity with the incompressibility condition
\begin{eqnarray*}
\nabla p=\nabla \Delta^{-1}{\rm div }\big(-v\cdot\nabla v_1+F(\theta_{1})-F(\theta_{2}) \big)-\nabla\Delta^{-1}{\rm div }(v_2\cdot\nabla v).
\end{eqnarray*}
Since ${\rm div } (v_{2}\cdot \nabla v)= {\rm div } (v\cdot \nabla v_2),$ then
$$\nabla p=\nabla\Delta^{-1}{\rm div }\left(-v\cdot\nabla (v_1+v_2)+F(\theta_{1})-F(\theta_{2})\right).$$
Using the continuity of Riesz transform on $L^{p}$ with $1<p<\infty,$ we get 
$$\|\nabla p\|_{L^p}\lesssim \|v\|_{L^p}\left(\|\nabla v_1\|_{L^\infty}+\|\nabla v_2\|_{L^\infty}\right)+\|F(\theta_{1})-F(\theta_{2})\|_{L^p}.$$
Combining this estimate with the $L^p$ estimate of the velocity we get
$$\|v(t)\|_{L^p}\lesssim \|v^0\|_{L^p}+\int_0^t\|v(\tau)\|_{L^p}(\big\|\nabla v_1(\tau)\|_{L^\infty}+\|\nabla v_2(\tau)\|_{L^\infty}\big)d\tau+\|F(\theta_{1})-F(\theta_{2})\|_{L^1_t L^p}.$$
Let us now estimate $\Vert F(\theta_{1})-F(\theta_{2}) \Vert_{L^{1}_{t}L^{p}}.$ Applying Taylor formula at order 1,
$$F(\theta_{1})-F(\theta_{2})= (\theta_{1}-\theta_{2})\int^{1}_{0} F^{\prime}\left( \theta_{2}+s\,(\theta_{1}-\theta_{2})\right) \,ds.$$
Taking the $L^{p}$ norm yields
$$\Vert F(\theta_{1})-F(\theta_{2}) \Vert_{L^{p}} \leqslant \Vert \theta \Vert_{L^{p}} \int^{1}_{0} \Vert F^{\prime}(\theta_{2}+s\,(\theta_{1}-\theta_{2})) \Vert_{L^{\infty}}\,ds.$$
Now we write
$$\Vert  F^{\prime}(\theta_{2}+s\,(\theta_{1}-\theta_{2})) \Vert_{L^{\infty}}\le \displaystyle \sup_{\vert x \vert \le C \Vert \theta^{0}\Vert_{L^{\infty}}} \vert F^{\prime} (x) \vert \le C.$$
Therefore
$$\Vert  F(\theta_{1})-F(\theta_{2}) \Vert_{L_{t}^{1}L^{p}}\lesssim \Vert \theta \Vert_{L^{1}_{t}L^{p}}.$$
Thus we obtain
$$\|v(t)\|_{L^p}\lesssim \|v^0\|_{L^p}+\int_0^t\|v(\tau)\|_{L^p}\left(\|\nabla v_1(\tau)\|_{L^\infty}+\|\nabla v_2(\tau)\|_{L^\infty}\right) d\tau+\|\theta\|_{L^1_t L^p}.$$
At this stage, we need to split $\theta$ into two parts $\theta=\widetilde{\theta}_{1}+\widetilde{\theta}_{2}$ where $\widetilde{\theta}_{1}$ and $\widetilde{\theta}_{2}$ solve respectively the following equations
\begin{equation} \label{T20}
\left\{\begin{array}{ll} 
\partial_t\widetilde{\theta}_{1}+v_2\cdot\nabla \widetilde{\theta}_{1}+\vert \DD \vert^{\alpha}\widetilde{\theta}_{1}=-v\cdot\nabla\theta_1\\
\widetilde{\theta_1}_{| t=0}=0. 
\end{array} \right. 
\end{equation} 
and
\begin{equation}\label{T21} 
\left\{\begin{array}{ll} 
\partial_t\widetilde{\theta}_{2}+v_2\cdot\nabla \widetilde{\theta}_{2}+\vert \DD \vert^{\alpha}\widetilde{\theta}_{2}=0\\
\widetilde{\theta_2}_{| t=0}=\theta^0. 
\end{array} \right. 
\end{equation} 
Taking the $L^{p}$ norm of \eqref{T20} we obtain from Lemma \ref{lem3},
\begin{eqnarray*}
\Vert \widetilde{\theta}_{1}(t) \Vert_{L^{p}}\le \displaystyle \int^{t}_{0}\Vert v\cdot\nabla\theta_{1}(\tau)\Vert_{L^{p}}d\tau 
&\le & \displaystyle \int^{t}_{0}\Vert v(\tau)\Vert_{L^{p}}\Vert \nabla \theta_{1}(\tau)\Vert_{L^{\infty}}d\tau.
\end{eqnarray*}
Integrating in time, we get
\begin{equation*}
\Vert \widetilde{\theta}_{1}(t) \Vert_{L^1_tL^{p}}\le t \int^{t}_{0}\Vert v(\tau)\Vert_{L^{p}}\Vert \nabla \theta_{1}(\tau)\Vert_{L^{\infty}}d\tau.
\end{equation*}
Now, we apply the operator $\Delta_{q}$ to the equation \eqref{T21} we have 
\begin{equation*}
\partial_{t}\Delta_{q}\displaystyle \widetilde{\theta}_{2}+v_{2}\cdot\nabla \Delta_{q}\widetilde{\theta}_{2}+\vert \DD\vert^{\alpha}\Delta_{q}\widetilde{\theta}_{2}=-[\Delta_{q},v_{2}\cdot\nabla] \widetilde{\theta}_{2}.
\end{equation*}
Taking the $L^{p}$ norm of the above equation and using Proposition \ref{prop5}, Lemma \ref{lem3} and Lemma \ref{lem4}, we obtain for $q\in\NN,$
\begin{eqnarray*}\label{d29}
\Vert \Delta_{q}\widetilde{\theta}_{2}\Vert_{L^{1}_{t}L^{p}}&\lesssim& 2^{-q \alpha}\Vert \Delta_{q}\theta^{0}\Vert_{L^{p}}+2^{-q \alpha}\int^{t}_{0}\Vert \nabla v_{2}(\tau)\Vert_{L^{p}}\Vert \widetilde{\theta}_{2}(\tau) \Vert_{L^{\infty}} d\tau\\
&\lesssim& 2^{-q \alpha}\Vert \Delta_{q}\theta^{0}\Vert_{L^{p}}+2^{-q \alpha}\Vert \theta^{0}\Vert_{L^{\infty}}\int^{t}_{0}\Vert \nabla v_{2}(\tau)\Vert_{L^{p}}d\tau.
\end{eqnarray*}
Summing these estimates on $q\geqslant -1$ and using Lemma \ref{lem3} and Proposition \ref{prop5}, we find
\begin{eqnarray*}\label{d30}
\sum_{q \geqslant -1}\Vert \Delta_{q}\widetilde{\theta}_{2}\Vert_{L^{1}_{t}L^{p}}&\lesssim& \sum_{q \geqslant 0}2^{-q \alpha}\Vert \Delta_{q}\theta^{0}\Vert_{L^{p}}+\sum_{q \geqslant 0}2^{-q \alpha}\Vert \theta^{0}\Vert_{L^{\infty}}\int^{t}_{0}\Vert \nabla v_{2}(\tau)\Vert_{L^{p}}d\tau+\Vert \Delta_{-1}\widetilde{\theta}_{2}\Vert_{L^{1}_{t}L^{p}}\\
&\lesssim& \displaystyle \sum_{q \geqslant 0}2^{-q \alpha}\Vert \Delta_{q}\theta^{0}\Vert_{L^{p}}+\Vert \theta^{0}\Vert_{L^{\infty}}\int^{t}_{0}\Vert \nabla v_{2}(\tau)\Vert_{L^{p}}d\tau\\
&+&t(\Vert \Delta_{-1}\theta^{0}\Vert_{L^{p}}+\Vert \theta^{0}\Vert_{L^{\infty}}\displaystyle \int^{t}_{0}\Vert \nabla v_{2}(\tau)\Vert_{L^{p}}d\tau)\\
&\lesssim& \Vert \theta^{0}\Vert_{B^{-\alpha}_{p,1}}+\Vert \theta^{0}\Vert_{L^{\infty}}\Vert \nabla v_{2}\Vert_{L^{1}_{t}L^{p}}. 
\end{eqnarray*}
 Therefore
\begin{equation*}
\Vert \widetilde{\theta}_{2}\Vert_{L^{1}_{t}L^{p}}\le \sum_{q\geqslant -1}\Vert \Delta_{q}\widetilde{\theta}_{2}\Vert_{L^{1}_{t}L^{p}}\lesssim \Vert \theta^{0}\Vert_{B^{-\alpha+1+\frac{2}{p}}_{p,1}\cap L^{\infty}}(1+\Vert \nabla v_{2}\Vert_{L^{1}_{t}L^{p}}).
\end{equation*}
We have used the Besov embedding $B^{-\alpha+1+\frac{2}{p}}_{p,1}\hookrightarrow B^{-\alpha}_{p,1}.$ Now since $\theta=\widetilde{\theta}_{1}+\widetilde{\theta}_{2},$ then we have
\begin{eqnarray*}
\Vert \theta \Vert_{L^{1}_{t}L^{p}} 
\le t \displaystyle \int^{t}_{0}\Vert v(\tau) \Vert_{L^{p}}\Vert \nabla \theta_{1}(\tau)\Vert_{L^{\infty}}d\tau +\displaystyle \Vert \theta^{0}\Vert_{B^{-\alpha+1+\frac{2}{p}}_{p,1}\cap L^{\infty}}(1+\int^{t}_{0}\Vert \nabla v_2(\tau) \Vert_{L^{p}}d\tau).
\end{eqnarray*}
Combining this estimate with the $L^{p}$ norm of the velocity, we find
\begin{eqnarray*}\label{d31}
\Vert v(t) \Vert_{L^{p}}&\lesssim& \Vert v^{0}\Vert_{L^{p}}+\int^{t}_{0}\Vert v(\tau) \Vert_{L^{p}}\big(\Vert \nabla v_{1}(\tau) \Vert_{L^{\infty}}+\Vert \nabla v_{2}(\tau) \Vert_{L^{\infty}}\big)d\tau\\
&+&t\displaystyle \int^{t}_{0}\Vert v(\tau)\Vert_{L^{p}}\Vert \nabla \theta_{1}(\tau) \Vert_{L^{\infty}}d\tau
+\Vert \theta^{0}\Vert_{B^{-\alpha+1+\frac{2}{p}}_{p,1}\cap L^{\infty}}(1+\displaystyle \int^{t}_{0}\Vert \nabla v_{2} \Vert_{L^{p}} d\tau).
\end{eqnarray*}
Finally we get by Gronwall's inequality, 
\begin{eqnarray}\label{T22}
\nonumber \|v(t)\|_{L^p}&\lesssim& e^{C(\Vert \nabla v_{1}\Vert_{L^1_t L^{\infty}}+\|\nabla v_2\|_{L^1_t L^{\infty}})}e^{t \Vert \nabla \theta_{1}\Vert_{L^{1}_{t}L^{\infty}}}\big(\|v^0\|_{L^p}\\
&+&\|\theta^0\|_{B^{-\alpha+1+\frac{2}{p}}_{p,1}\cap L^\infty}\Vert\nabla v_{2}\Vert_{L^{1}_{t}L^{p}}\big).
\end{eqnarray}
This gives in turn 
\begin{eqnarray}\label{T23}
\nonumber \Vert \theta(t)\Vert_{L^{1}_{t}L^p}&\lesssim& e^{C(\Vert \nabla v_{1}\Vert_{L^1_t L^{\infty}}+\|\nabla v_2\|_{L^1_t L^{\infty}})}e^{t \Vert \nabla \theta_{1}\Vert_{L^{1}_{t}L^{\infty}}}\big(\|v^0\|_{L^p}\\
&+&\|\theta^0\|_{B^{-\alpha+1+\frac{2}{p}}_{p,1}\cap L^\infty}\Vert \nabla v_{2}\Vert_{L^{1}_{t}L^{p}}\big)(t+1).\end{eqnarray}
The proof of the uniqueness part is now complete.

\subsection{Existence.} Let us now outline briefly the proof of the existence of global solution to \eqref{T1} First we need to the following lemma (see \cite{AH} for the proof).
\begin{lem}\label{lem6}
Let $s \in \RR\;,\,(p,r) \in [1,\infty[^{2}$ and $G \in B^{s}_{p,r}(\RR^d).$ Then there exists $G^{n}\in\mathcal{S}(\RR^{d})$ such that for all $\varepsilon >0$ there exist $n_{0}$ such that 
$$\Vert G^{n}-G \Vert_{B^{s}_{p,r}}\le \varepsilon\;,\;\forall\; n\geqslant n_{0}.$$
If in addition $G\in L^{\infty}(\RR^{d}),$ then
$$\Vert G^{n}\Vert_{L^{\infty}(\RR^{d})}\lesssim \Vert G \Vert_{L^{\infty}(\RR^{d})}.$$ 
And if $\textnormal{div}\,G=0$ then $\textnormal{div}\,G^{n}=0.$
\end{lem}
We consider the following system
\begin{equation}\label{T24} 
\left\{\begin{array}{ll} 
\partial_{t}v_n+v_n\cdot\nabla v_n+\nabla p_n=F(\theta_n)\\ 
\partial_{t}\theta_n+v_n\cdot\nabla\theta_n+\vert \textnormal{D}\vert^\alpha\theta_n=0\\
\textnormal{div}\,v_n=0\\
{v_n}_{| t=0}=v_{n,0}\;\;,\quad {\theta_n}_{| t=0}=\theta_{n,0}.  
\end{array} \right.
\end{equation}
By using the same method of \cite{hk1}, we can prove that this system has a unique local solution $(v_{n},\theta_{n}).$ The global existence of these solutions is governed by $V_{n}$ where
$$V_{n}(t)=\int^{t}_{0}\Vert \nabla v_{n}(\tau)\Vert_{L^{\infty}}d\tau \le \int^{t}_{0}C_{0}e^{\exp\,C_{0}\tau}d\tau \le C_{0}e^{\exp\,C_{0}t}.$$ 
Now from the a priori estimates the Lipschitz norm can not blow up in finite time and then the solution $(v_{n},\theta_{n})$ is globally defined. Once again from the a priori estimates we have
$$\|\theta_{n}\|_{\widetilde L^\rho_TB^{\frac{\alpha}{\rho}}_{\infty,\infty}}+\|v_n\|_{ L^\infty_TB_{p,1}^{1+\frac2p}}+\|\theta_n\|_{L^\infty_T\big( L^\infty\cap B_{p,1}^{-\alpha+1+\frac2p}\big)}\leq C_{0}e^{e^{\exp C_{0}T}}.$$
Then it follows that up to an extraction the sequence $(v_{n},\theta_{n})$ is weakly convergent to $(v,\theta)$ belonging to $L^{\infty}_{T}B_{p,1}^{1+\frac{2}{p}}\times L^{\infty}_{T}(L^{\infty} \cap B_{p,1}^{-\alpha+1+\frac{2}{p}})\cap \widetilde{L}^{\rho}_{T}B^{\frac{\alpha}{\rho}}_{\infty,\infty}.$\\
We will now prove that the sequence $(v_{n},\theta_{n})$ is of a Cauchy in $L^\infty_TL^{ p}\times L^1_TL^{p}.$\\
Let $(n,n_1)\in\NN^2$ , $v_{n,n_1}=v_n-v_{n_{1}}$ and $\theta_{n,n_1}=\theta_n- \theta_{n_{1}}$ then  according to the estimates \eqref{T22} and \eqref{T23}, we get  
$$\|v_{n,n_1}\|_{L^\infty_TL^{p}}+\|\theta_{n,n_1}\|_{L^1_TL^{p}}\leq C_{0}e^{e^{\exp C_{0}T}} \big( \|v_{n,0}- v_{n_{1},0}\|_{L^p}+\|\theta_{n,0}-\theta_{n_{1},0}\|_{B^{-\alpha+1+\frac{2}{p}}_{p,1}\cap L^{\infty}}\big).$$
This show  that the sequence  $(v_n,\theta_n)$ is of a Cauchy in the Banach space $ L^\infty_TL^{p}\times L^1_TL^{p}.$ Hence it converges strongly to $(v,\theta).$ This allows us to pass to the limit in the system \eqref{T24} and then we get that $(v,\theta)$ is a  solution  of the system \eqref{T1}.\\

Let us now sketch the proof of the continuity in time of the velocity. From the definition of Besov space we have for $N \in\NN\;,\;T>0$ and for $t,t^\prime \in \RR_+,$
\begin{eqnarray}\label{T25}
\nonumber\Vert v(t)-v(t^\prime) \Vert_{B^{1+\frac{2}{p}}_{p,1}}&\le& \sum_{q \le N}2^{q(1+\frac{2}{p})}\Vert \Delta_{q}\big(v(t)- v(t^\prime) \big)\Vert_{L^{p}}+2 \sum_{q > N}2^{q(1+\frac{2}{p})}\Vert \Delta_{q}v \Vert_{L^{p}}\\
&\le& C 2^{N(1+\frac{2}{p})} \Vert v(t)-v(t^\prime) \Vert_{L^{p}}+2 \displaystyle \sum_{q > N}2^{q(1+\frac{2}{p})}\Vert \Delta_{q}v \Vert_{L_T^\infty L^{p}}.
\end{eqnarray}
It remains then to estimate $\Vert v(t)-v(t^\prime)\Vert_{L^p}.$ For this purpose we use the velocity equation
$$\partial_{t}v=-\mathcal{P}(v\cdot\nabla v)+ \mathcal{P}F(\theta).$$
Where $\mathcal{P}$ denotes Leray projector. The solution of this equation is given by Duhamel formula
$$v(t,x)=v^{0}(x)-\int^{t}_{0}\mathcal{P}(v\cdot\nabla v)(\tau) d\tau +\int^{t}_{0}\mathcal{P}(F(\theta(\tau))) d\tau.$$
Hence it follows that for $t,t^\prime \in \RR_+,$ 
$$v(t,x)-v(t^\prime,x)=-\int^{t}_{t^\prime}\mathcal{P}(v\cdot\nabla v)(\tau) d\tau +\int^{t}_{t^\prime}\mathcal{P}(F(\theta(\tau))) d\tau.$$
Taking the $L^p$ norm of the above equation and using the fact that the Leray projector $\mathcal{P}$ is continuously into $L^p,$ with $1 < p < \infty$ we get then 
\begin{eqnarray}\label{T26}
\nonumber\Vert v(t)-v(t^\prime) \Vert_{L^p}&\lesssim& \int^{t}_{t^\prime}\Vert (v\cdot\nabla v)(\tau) \Vert_{L^p}d\tau +\int^{t}_{t^\prime}\Vert F(\theta(\tau))\Vert_{L^p}d\tau\\
\nonumber &\lesssim& \displaystyle \int^{t}_{t^\prime}\Vert v(\tau)\Vert_{L^{p}}\Vert \nabla v(\tau)\Vert_{L^{\infty}}d\tau+\int^{t}_{t^\prime}\Vert F(\theta(\tau)) \Vert_{L^p}d\tau\\
&\lesssim& \vert t-t^\prime \vert \Vert v \Vert_{L^{\infty}_{t}L^{p}}\Vert \nabla v \Vert_{L^{\infty}_{t}L^{\infty}}+\displaystyle \int^{t}_{t^\prime}\Vert F(\theta(\tau)) \Vert_{L^p}d\tau.
\end{eqnarray}
We have used H\"older's inequality and integration by parts for the first term of the above inequality. For the last term we use Taylor formula with $F$ vanishing at 0,
\begin{equation*}
F(\theta) =\theta \displaystyle \int^{1}_{0}F^{\prime}(s \theta) ds.
\end{equation*}
Thus
$$\Vert F(\theta) \Vert_{L^{p}}\le \Vert \theta \Vert_{L^{p}}\displaystyle \int^{1}_{0}\Vert F^{\prime}(s \theta) \Vert_{L^{\infty}} ds$$
Now,
\begin{equation*}
\Vert F^{\prime}(s \theta) \Vert_{L^{\infty}}\lesssim \sup_{\vert y \vert \le \Vert \theta^{0}\Vert_{L^{\infty}}}\vert F^{\prime}(y)\vert \le C.
\end{equation*}
Hence it follows that $$\Vert F(\theta) \Vert_{L^{p}}\lesssim \Vert \theta \Vert_{L^{p}},$$
which yields
\begin{equation*}
\displaystyle \int^{t}_{t^\prime}\Vert F(\theta(\tau))\Vert_{L^{p}}d\tau \lesssim \displaystyle \int^{t}_{t^\prime}\Vert \theta(\tau)\Vert_{L^{p}}d\tau\\
\lesssim \vert t-t^\prime \vert^{\frac{1}{2}} \Vert \theta \Vert_{L^2_{t}L^{p}}.
\end{equation*}
We use now Proposition \ref{prop5} with $\rho=2,$
\begin{eqnarray*}\label{d35}
\Vert \theta \Vert_{L^{2}_{t}L^{p}}&\le& \sum_{q \geqslant -1}\Vert \Delta_{q}\theta \Vert _{L^{2}_{t}L^{p}}\\
&\lesssim& \displaystyle \sum_{q \geqslant 0}2^{-q \frac{\alpha}{2}}\big(\Vert \Delta_{q}\theta^{0}\Vert_{L^{p}}+\Vert \theta^{0}\Vert_{L^{\infty}}\Vert \nabla v \Vert_{L^{1}_{t}L^{p}}\big)\\
&+& t^{\frac{1}{2}}(\Vert \Delta_{-1}\theta^{0}\Vert_{L^{p}}+\Vert \theta^{0}\Vert_{L^{\infty}}\Vert \nabla v \Vert_{L^{1}_{t}L^{p}})\\
&\lesssim& \Vert \theta^{0}\Vert_{B^{\frac{-\alpha}{2}}_{p,1}}+\Vert \theta^{0}\Vert_{L^{\infty}}\Vert \omega \Vert_{L^{1}_{t}L^{p}}\\
&\lesssim& \Vert \theta^{0}\Vert_{B^{-\alpha+1+\frac{2}{p}}_{p,1}\cap L^{\infty}}(1+\Vert \omega \Vert_{L^{1}_{t}L^{p}}),
\end{eqnarray*}
we have used the embedding $B^{-\alpha+1+\frac 2p}_{p,1}\hookrightarrow B^{\frac{-\alpha}{2}}_{p,1}$ (recall that $\alpha\le 2$). Therefore
$$\int^{t}_{t^\prime}\Vert F(\theta(\tau))\Vert_{L^{p}}d\tau \lesssim \vert t-t^\prime \vert^{\frac{1}{2}}\Vert \theta^{0}\Vert_{B^{-\alpha+1+\frac{2}{p}}_{p,1}\cap L^{\infty}}(1+\Vert \omega \Vert_{L^{1}_{t}L^{p}}).$$
Finally we obtain in \eqref{T26}
\begin{equation*}
\Vert v(t)-v(t^{\prime}) \Vert_{L^{p}}\lesssim \vert t-t^{\prime}\vert \Vert v \Vert_{L^{\infty}_{t}L^{p}}\Vert \nabla v \Vert_{L^{\infty}_{t}L^{\infty}}+\vert t-t^{\prime}\vert^{\frac{1}{2}}\Vert \theta^{0}\Vert_{B^{-\alpha+1+\frac{2}{p}}_{p,1}\cap L^{\infty}}(1+\Vert \omega \Vert_{L^{1}_{t}L^{p}}).
\end{equation*}
Combining Proposition \ref{prop7} and Proposition \ref{prop8} with the inequalities \eqref{T25} and the previous, we obtain
\begin{eqnarray}\label{T27}
\nonumber \Vert v(t)-v(t^\prime) \Vert_{B_{p,1}^{1+\frac{2}{p}}}&\lesssim& 2^{N(1+\frac{2}{p})}\bigg(\vert t-t^{\prime}\vert \Vert v \Vert_{L^{\infty}_{t}L^{p}}C_{0}e^{e^{C_{0}t}}+\vert t-t^{\prime}\vert^{\frac{1}{2}}\Vert \theta^{0}\Vert_{B^{-\alpha+1+\frac{2}{p}}_{p,1}\cap L^{\infty}}C_{0}e^{C_{0}t}\bigg)\\
&+& 2 \sum_{q > N}2^{q(1+\frac{2}{p})}\Vert \Delta_{q}v \Vert_{L_T^\infty L^{p}}.
\end{eqnarray}
We have by Proposition \ref{prop9} that $v \in \widetilde{L}^{\infty}_{T}B^{1+\frac{2}{p}}_{p,1},$ then for $\varepsilon >0,$ there exists an integer $N$ such that
$$\displaystyle \sum_{q > N}2^{q(1+\frac{2}{p})}\Vert \Delta_{q}v \Vert_{L_T^\infty L^{p}}\le \frac{\varepsilon}{4}.$$
It is enough to choose $\vert t-t^\prime \vert <\eta$ such that 
$$2^{N(1+\frac{2}{p})}\bigg(\vert t-t^{\prime}\vert \Vert v \Vert_{L^{\infty}_{t}L^{p}}C_{0}e^{e^{C_{0}t}}+\vert t-t^{\prime}\vert^{\frac{1}{2}}\Vert \theta^{0}\Vert_{B^{-\alpha+1+\frac{2}{p}}_{p,1}\cap L^{\infty}} C_{0}e^{C_{0}t}\bigg)<\frac{\varepsilon}{2}.$$
Finally we find in \eqref{T27} that
$$\Vert v(t)-v(t^\prime) \Vert_{B_{p,1}^{1+\frac{2}{p}}}\le \varepsilon.$$
This proves the continuity in time of the velocity.
\subsection{Appendix: Generalized Bernstein inequality}
The generalized Bernstein inequality is proved in \cite{cmz, D} for $0<\alpha\le2$ and $p\ge 2.$ Here we extend this inequality for the remaining case $p\in ]1,2]$. More precisely,  we have the following proposition.
\begin{prop}\label{prop10}
We assume that $\alpha \in]0,1]$ and $p>1.$ Then we have for every $G \in \mathcal{S}(\RR^{2})$ and $j\in\NN,$
\begin{equation*}
c 2^{j\alpha}\Vert\Delta_{j}G\Vert^{p}_{L^{p}}\le \displaystyle\int_{\RR^{2}}(\vert\DD\vert^{\alpha} \Delta_{j}G) \vert \Delta_{j}G\vert^{p-1} sign\; \Delta_{j}G\,dx,
\end{equation*}
where $c$ depend on $p.$
\end{prop}
\begin{proof}
We use the following Corollary (see \cite{N} for the proof).
\begin{cor}\label{coro}
Let $\alpha\in]0,1]$ and $p>1.$ Then we have,
\begin{equation*}
4\frac{p-1}{p^{2}}\Vert \vert \DD \vert^\frac{\alpha}{2} \vert G \vert^\frac{p}{2}\Vert^{2}_{L^{2}}\le \int_{\RR^{2}}(\vert \DD \vert^{\alpha} G) \vert G \vert^{p-1} sign\,G\;dx.
\end{equation*}
\end{cor}
It suffices thus to prove,
\begin{equation*}    
c 2^{j \alpha}\Vert\Delta_{j}G\Vert^{p}_{L^{p}}\le \Vert \vert \DD \vert^\frac{\alpha}{2}(\vert \Delta_{j}G \vert^\frac{p}{2})\Vert^{2}_{L^{2}}.
\end{equation*}
Let $N \in\NN,$ we define $\Delta_{j}G:=G_{j},$ then 
\begin{equation}\label{T28}
\Vert \vert \DD \vert(\vert G_{j}\vert^\frac{p}{2})\Vert_{L^{2}}\le \Vert S_{N} \vert \DD \vert(\vert G_{j}\vert^\frac{p}{2})\Vert_{L^{2}}+ \Vert (Id-S_{N}) \vert \DD \vert(\vert G_{j}\vert^\frac{p}{2})\Vert_{L^{2}}. 
\end{equation}
Let $s^{\prime}>0,$ then Bernstein inequality gives,
\begin{eqnarray}\label{T29}
\nonumber \big\Vert (Id-S_N) \vert \DD \vert(\vert G_j\vert^\frac{p}{2})\big\Vert_{L^{2}}&\le& \sum_{k \geq N}\big\Vert \Delta_k(\vert \DD \vert(\vert G_j\vert^\frac{p}{2}))\big\Vert_{L^{2}}\\
\nonumber &\lesssim& \sum_{k \geq N}2^{-k s^{\prime}}2^{k(1+s^{\prime})}\big\Vert \Delta_{k}(\vert G_j\vert^\frac{p}{2})\big\Vert_{L^{2}}\\
\nonumber &\lesssim& 2^{-N\,s^{\prime}}\big\Vert \vert G_{j}\vert^\frac{p}{2}\big\Vert_{B^{1+s^{\prime}}_{2,\infty}}\\
&\lesssim& 2^{-N\,s^{\prime}}\big\Vert \vert G_j \vert^\frac{p}{2}\big\Vert_{H^{1+s^{\prime}}}.
\end{eqnarray}
we have used in the last line the Besov embedding $H^{1+s^{\prime}}\hookrightarrow B^{1+s^{\prime}}_{2,\infty}.$\\
To estimate $\Vert \vert G_{j}\vert^\frac{p}{2} \Vert_{H^{1+s^{\prime}}},$ we will use the following Lemma.
\begin{lem}\label{lem7}
$1)$ Let $\gamma \geqslant 1$ and $s^{\prime} \in [0,\gamma[ \cap [0,2[.$ Then
$$\Vert \vert G \vert^{\gamma} \Vert_{H^{s^{\prime}}} \lesssim \Vert G \Vert_{B^{s^{\prime}}_{2 \gamma,2}} \Vert G \Vert^{\gamma-1}_{B^{0}_{2\gamma,2}},$$
$2)$ For $0 < \gamma \le 1\;,\,(p,r) \in [1,\infty]^{2}$ and $0 < s^{\prime} < 1+\frac{1}{p}.$ Then\\
$$\Vert \vert G \vert^{\gamma} \Vert_{B^{s^{\prime}\gamma}_{\frac{p}{\gamma},\frac{r}{\gamma}}} \lesssim \Vert G \Vert^{\gamma}_{B^{s^{\prime}}_{p,r}}.$$
\end{lem}
The first estimate is a particular case of a general result due to \cite{cmz}. The second is established by Sickel \cite{S} (see also Theorem 1.4 of \cite{K}).\\
We use Lemma \ref{lem7}-1) and Bernstein inequality for  $p > 2$, with $0<s^{\prime}<min(\frac{p}{2}-1,2),$
\begin{eqnarray*}\label{d39}
\Vert \vert G_{j} \vert^{\frac{p}{2}} \Vert_{H^{1+s^{\prime}}}&\lesssim& \Vert G_{j} \Vert_{B^{1+s^{\prime}}_{p,2}} \Vert G_{j} \Vert^{\frac{p}{2}-1}_{B^{0}_{p,2}}\\
&\lesssim& 2^{j(1+s^{\prime})} \Vert G_{j} \Vert^{\frac{p}{2}}_{L^{p}}.
\end{eqnarray*}
For $1<p \le 2 :$ we use Lemma \ref{lem7}-2) and Bernstein inequality with $0 < s^{\prime}<1+\frac{1}{p},$
\begin{eqnarray*}\label{d40}
\Vert \vert G_{j} \vert^{\frac{p}{2}} \Vert_{H^{1+s^{\prime}}}&\lesssim& \Vert G_{j} \Vert^{\frac{p}{2}}_{B^{\frac{2(1+s^{\prime})}{p}}_{p,p}}\\
&\lesssim& 2^{j(1+s^{\prime})} \Vert G_{j} \Vert^{\frac{p}{2}}_{L^{p}}.
\end{eqnarray*}
We deduce thus from \eqref{T29} and Lemma \ref{lem7},
\begin{equation}\label{T30}
\Vert (Id-S_{N}) \vert \DD \vert(\vert G_{j}\vert^\frac{p}{2})\Vert_{L^{2}}\le 2^{-N\,s^{\prime}}2^{j(1+s^{\prime})}\Vert G_{j}\Vert_{L^{p}}^\frac{p}{2}.
\end{equation}
To estimate the first norm of \eqref{T28}, we use Bernstein inequality,
\begin{eqnarray}\label{T31}
\nonumber\Vert S_{N} \vert \DD \vert(\vert G_{j}\vert^\frac{p}{2})\Vert_{L^{2}}&\lesssim& \Vert S_{N} \vert \DD \vert^{1-\frac{\alpha}{2}}(\vert \DD \vert^{\frac{\alpha}{2}}(\vert G_{j}\vert^\frac{p}{2}))\Vert_{L^{2}}\\
\nonumber&\lesssim& 2^{N(1-\frac{\alpha}{2})} \Vert S_{N} \vert \DD \vert^{\frac{\alpha}{2}}(\vert G_{j}\vert^\frac{p}{2})\Vert_{L^{2}}\\
&\lesssim& 2^{N(1-\frac{\alpha}{2})} \Vert \vert \DD \vert^{\frac{\alpha}{2}}(\vert G_{j}\vert^\frac{p}{2})\Vert_{L^{2}}.
\end{eqnarray}
Putting \eqref{T30} and \eqref{T31} into \eqref{T28}, we get
\begin{equation*}
\Vert \vert \DD \vert(\vert G_{j}\vert^\frac{p}{2})\Vert_{L^{2}}\lesssim 2^{N(1-\frac{\alpha}{2})} \Vert \vert \DD \vert^{\frac{\alpha}{2}}(\vert G_{j}\vert^\frac{p}{2})\Vert_{L^{2}}
 + 2^{-N\,s^{\prime}}2^{j(1+s^{\prime})}\Vert G_{j} \Vert^\frac{p}{2}_{L^{p}}. 
\end{equation*}
According to Lemma A.5 of \cite{D}, we have for $1< p < \infty,$
$$c_{p}\, 2^{j} \Vert G_{j} \Vert^{\frac{p}{2}}_{L^{p}} \le \Vert \vert \DD \vert (\vert G_{j} \vert^{\frac{p}{2}} \Vert_{L^{2}}.$$
Combining both last estimates we get,
\begin{equation*}
c_{p}2^{j} \Vert G_{j} \Vert^{\frac{p}{2}}_{L^{p}} \le 2^{N(1-\frac{\alpha}{2})} \Vert \vert \DD \vert^{\frac{\alpha}{2}}(\vert G_{j}\vert^\frac{p}{2})\Vert_{L^{2}} + 2^{s^{\prime}(j-N)}2^{j}\Vert G_{j} \Vert^\frac{p}{2}_{L^{p}}. 
\end{equation*}  
Taking $N-j=N_{1}$ such that $2^{-N_{1}\,s^{\prime}} \le \frac{1}{2}.$ Therefore
\begin{equation*}
c2^{j\,\alpha} \Vert G_{j} \Vert^{p}_{L^{p}} \le \Vert \vert \DD \vert^{\frac{\alpha}{2}}(\vert G_{j}\vert^\frac{p}{2})\Vert^{2}_{L^{2}}.
\end{equation*}
with $c$ depend on $p.$ This proves the Proposition.
\end{proof}

\end{document}